\documentclass[reqno,11pt]{amsart}
\usepackage{amsfonts}
\usepackage{a4wide}
\usepackage{color}
\usepackage{mathrsfs}
\usepackage{mathtools}
\usepackage{amsmath}
\usepackage{amssymb}
\usepackage{bbm}
\usepackage{esint}
\usepackage{nicefrac}
\numberwithin{equation}{section}
\usepackage[colorlinks,citecolor=green,linkcolor=red]{hyperref}

\usepackage[latin1]{inputenc}
\input xy
\xyoption{all}

\newtheorem{theorem}{Theorem}[section]
\newtheorem{Ca}[theorem]{Corollary}
\newtheorem{Th}[theorem]{Theorem}
\newtheorem{Lm}[theorem]{Lemma}
\newtheorem{Prop}[theorem]{Proposition}
\newtheorem{Def}[theorem]{Definition}

\newtheorem{Remark}[theorem]{Remark}
\newtheorem{Problem}[theorem]{Problem}

\title{A New insight into the Ambrosio--Reshetnyak approach to Sobolev maps}
\author{Alexander I. Tyulenev}
\address{Steklov Mathematical Institute of Russian academy of Sciences}
\email{tyulenev-math@yandex.ru,tyulenev@mi-ras.ru}
\begin{document}
\date{\today}
\allowdisplaybreaks
\begin{abstract}
Given $p \in (1,\infty)$, let $\operatorname{X}=(\operatorname{X},\rho,\mu)$ be a metric measure space such that $(\operatorname{X},\operatorname{d})$ is complete, the measure
$\mu$ is uniformly locally doubling and $\operatorname{X}$ supports a weak local $(1,p)$-Poincar\'e inequality.
Let $(\operatorname{Y},\operatorname{d},\underline{y})$ be a complete pointed metric space. We prove that the equivalence class (modulo coincidence $\mu$-a.e.) of
a Borel map $u:\operatorname{X} \to \operatorname{Y}$ with a separable range belongs to the Ambrosio--Reshetnyak--Sobolev class $W_{p}^{1}(\operatorname{X},\operatorname{Y})$ if and only if $h \circ  u$ belongs to the Sobolev space $W_{p}^{1}(\operatorname{X})$ for every $1$-Lipschitz function $h:\operatorname{Y} \to \mathbb{R}$ with $h(\underline{y})=0$.
\end{abstract}
\maketitle
\tableofcontents
%
%

%
\section{Introduction}
\noindent

In his groundbreaking paper \cite{Amb} L.~Ambrosio built foundations of the theory of maps of bounded variation acting from a domain $\Omega \subset \mathbb{R}^{n}$ to a locally
compact metric space. A little bit later similar ideas
have been used by Yu.~G.~Reshetnyak in \cite{Re}, where the theory of Sobolev maps acting from a domain $\Omega \subset \mathbb{R}^{n}$ to an arbitrary complete metric space was developed.
However, it should be noted that probably the
first definition (different from that of used by Reshetnyak) of Sobolev metric-valued maps was given by N.~Korevaar and R.~Sch\"oen
in the famous paper \cite{KS}.

In \cite{GT} N.~Gigli and the author combined ideas from \cite{Amb,Re,KS} and introduced several approaches to the definitions of Sobolev maps in the fully nonsmooth setting. More precisely, we considered maps $ u: \operatorname{X} \to \operatorname{Y}$ assuming that $\operatorname{X}:=(\operatorname{X},\rho,\mu)$ is a metric measure space such that the measure $\mu$ is uniformly locally doubling and
$\operatorname{X}$ supports the corresponding Poincar\'e inequality, whereas $\operatorname{Y}:=(\operatorname{Y},\operatorname{d},\underline{y})$ is an arbitrary complete pointed metric space.
Under those assumptions on $\operatorname{X}$ and $\operatorname{Y}$
the equivalence between Korevaar--Sch\"oen's and Ambrosio--Reshetnyak's approaches was established in \cite{GT}.

The \textit{main goal} of the present paper is to deepen methods from \cite{GT}, combine them with new ``infinitesimal ideas'', and find a simplified and  - in some sense - unexpected
version of the Ambrosio--Reshetnyak approach to Sobolev maps under the same assumptions on $\operatorname{X}$ and $\operatorname{Y}$. This  gives an ``almost final'' step in the program initiated by R.~D.~Oleinik and the author \cite{OT}, where the similar breakthrough was made
in the particular case when $\operatorname{X}$ is the unit segment $[0,1]$ equipped with the Euclidean distance and the one-dimensional Lebesgue measure.

\subsection{Statement of the problem}
We briefly fix some terminology and refer the reader to Section 2 of this paper for the detailed notation and definitions.
We should warn the reader that, unless otherwise stated, given a metric measure space $\operatorname{X}=(\operatorname{X},\rho,\mu)$ and a pointed metric space $\operatorname{Y}=(\operatorname{Y},\operatorname{d},\underline{y})$, we  \textit{identify Borel maps} $ u:\operatorname{X} \to \operatorname{Y}$
with their \textit{equivalence classes} modulo coincidence $\mu$-almost everywhere on $\operatorname{X}$. In this case all inequalities and equalities containing
$ u$ should be interpreted for its arbitrary representative in the $\mu$-almost everywhere sense.

Let $\operatorname{X}:=(\operatorname{X},\rho,\mu)$ be a \textit{metric measure space} and let $\operatorname{Y}:=(\operatorname{Y},\operatorname{d},\underline{y})$ be a pointed \textit{complete metric space}.
Given $p \in [1,\infty]$, we say that a map $ u \in L_{p}(\operatorname{X},\operatorname{Y})$ belongs to the \textit{Ambrosio--Reshetnyak--Sobolev space} $W_{p}^{1}(\operatorname{X},\operatorname{Y})$ provided there is a nonnegative $G \in L_{p}(\operatorname{X})$ such that for each 1-Lipschitz function $h:\operatorname{Y} \to \mathbb{R}$ with $h(\underline{y})=0$ it holds that $h \circ  u \in W_{p}^{1}(\operatorname{X})$ and
\begin{equation}
\label{eqq.Amb_Resh_majorant}
|D(h\circ u)|_{p}(x) \le G(x) \quad \hbox{for $\mu$-a.e. $x \in \operatorname{X}$},
\end{equation}
where $|D(h\circ u)|_{p}$ is the minimal $p$-weak upper gradient of $h \circ u$.
Notice that there exists the minimal $G$ in the $\mu$-a.e. sense for which $\eqref{eqq.Amb_Resh_majorant}$ holds. We call it \textit{the minimal
$p$-weak upper gradient} of $u$ and denote it by $|Du|_{p}$. The index $p$ in our notation stresses the presence of dependence of
the $p$-minimal weak upper gradients on $p$ (see Remark \ref{Rem.dependence_on_p} for details). However, if $\operatorname{X}$ is a domain in a finite dimensional
Euclidean space, then one can show that $|Du|_{p}=|Du|_{q}$ in the $\mu$-a.e. sense for each $p,q \in (1,\infty)$ and any $u \in W_{p}^{1}(\operatorname{X},\operatorname{Y}) \cap W_{q}^{1}(\operatorname{X},\operatorname{Y})$.
In this case we will write $|Du|$ instead of $|Du|_{p}$.

The crucial point of the Ambrosio--Reshetnyak approach is the \textit{a priori assumption} concerning the existence of the majorant $G$ in \eqref{eqq.Amb_Resh_majorant}.
On the one hand, the requirement is quite natural. Indeed, in the smooth setting, i.e. when $\operatorname{X}$ is a finite-dimensional Euclidean space equipped with
the usual Lebesgue measure, $\operatorname{Y}$ is a finite dimensional Euclidean space, and $ u:\operatorname{X} \to \operatorname{Y}$ is smooth, the minimal weak upper gradient $|D u|$
is nothing but the operator norm of the classical differential of $ u$. On the other hand, in the smooth setting the fact of the presence of
a ``majorant'' $G$ follows automatically if one has at a disposal only qualitative information about the Sobolev regularity of compositions of the map with Lipschitz functions.
Moreover, one can extract the Sobolev regularity of a given map from its compositions with coordinate projections only. Since in the ``nonsmooth world''
typically there are no transparent ways to construct such majorants, the main question we are aimed to attack
in our paper can be informally written as follows: \textit{``Is the requirement \eqref{eqq.Amb_Resh_majorant} really necessary}?'' More precisely, we pose the following problem.
\begin{Problem}
\label{Problem.Oleinik}
Given $p \in [1,\infty)$, assume that $h \circ  u \in W_{p}^{1}(\operatorname{X})$ for every 1-Lipschitz functions $h:\operatorname{Y} \to \mathbb{R}$.
Is it true that $u \in W_{p}^{1}(\operatorname{X},\operatorname{Y})$?
\end{Problem}
This question is far from being elementary.
Indeed, as it will be clear during the rest of the paper, due to the lack of any smooth structure on the space $\operatorname{Y}$,
we should work with the whole family of $1$-Lipschitz functions on $\operatorname{Y}$, which is
typically uncountable, and there is no chance to fix some finite or even countable family of $1$-Lipschitz functions which could be used for the role of ``test functions''.

\subsection{Previous results} Throughout this subsection, given $n \in \mathbb{N}$, by $\mathcal{L}^{n}$ we denote the classical $n$-dimensional Lebesgue measure on $\mathbb{R}^{n}$.

The \textit{first step} tovards to Problem \ref{Problem.Oleinik} is presented in the joint paper of R.~D.~Oleinik and the author
\cite{OT} where the case $\operatorname{X}=([0,1],|\cdot|,\mathcal{L}^{1}\lfloor_{[0,1]})$ was completely studied.
More precisely, combining the main results in \cite{OT} we have the following result.
\begin{Prop}
\label{Prop.OT_2}
Let $\operatorname{Y}=(\operatorname{Y},\operatorname{d})$ be a complete separable metric space. Given $p \in (1,\infty)$,
a map $ u:[0,1] \to \operatorname{Y}$ belongs to the Ambrosio--Reshetnyak--Sobolev space $W_{p}^{1}([0,1],\operatorname{Y})$
if and only if, for each $1$-Lipschitz function $h:\operatorname{Y} \to \mathbb{R}$, the composition
$h \circ  u$ belongs to the Sobolev space $W_{p}^{1}([0,1])$. Furthermore, there exists a $1$-Lipschitz
function $\underline{h}:\operatorname{Y} \to \mathbb{R}$ such that
\begin{equation}
\label{eqq.OT_equality}
|D(\underline{h} \circ  u)|(x) = |D u|(x) \quad \hbox{for $\mathcal{L}^{1}$-a.e. $x \in [0,1]$}.
\end{equation}
\end{Prop}
In the context of $p$-absolutely continuous mappings the similar result was proved in \cite{OT} in the full range $p \in [1,\infty]$ without any assumptions on $\operatorname{Y}$.
Independently and almost simultaneously using quite elementary but beautiful geometric ideas V.~I.~Bakhtin proved \cite{Bah} a slightly weaker result in the particular case $p=1$.

The key ingredient for the proof of Proposition \ref{Prop.OT_2} is the characterization of the continuous metric curves of bounded variation in terms of Lipschitz post-compositions.
Given a metric space $\operatorname{Y}=(\operatorname{Y},\operatorname{d})$, by $BV([0,1],\operatorname{Y})$ we denote the space of all maps $ u:[0,1] \to \operatorname{Y}$ of bounded variation.
Once $\operatorname{Y}$ is the usual real line, we will omit it from the corresponding notation.
\begin{Prop}
\label{Prop.OT_1}
Let $\operatorname{Y}=(\operatorname{Y},\operatorname{d})$ be an arbitrary metric space. A map $ u:[0,1] \to \operatorname{Y}$ belongs to the space $BV([0,1],\operatorname{Y}) \cap C([0,1],\operatorname{Y})$
if and only if, for each $1$-Lipschitz function $h:\operatorname{Y} \to \mathbb{R}$, the composition $h \circ  u$ belongs to the space $BV([0,1]) \cap C([0,1])$. Moreover, there is a $1$-Lipschitz function $\underline{h}:\operatorname{X} \to \mathbb{R}$
such that for any segment $[c,d] \subset [0,1]$ the total variation $V_{ u}([c,d])$ of the map $ u$ on $[c,d]$ coincides with the total variation $V_{\underline{h} \circ  u}([c,d])$ of the function
$\underline{h} \circ  u$ on $[c,d]$.
\end{Prop}
It turns out that the continuity assumption is crucial for the possibility of the characterization of $BV$-maps via post-compositions with Lipschitz functions.
It was shown in \cite{ST} (see also \cite{ST2} for the short announcement and \cite{S} for the further development) that there are many geometric obstacles of different nature that
do not allow to obtain any reasonable analog of Proposition \ref{Prop.OT_1} for generic metric spaces $(\operatorname{Y},\operatorname{d})$ and discontinuous maps
$ u:[0,1] \to \operatorname{Y}$.

The \textit{second progress} was achieved in \cite{PavVod} where Problem \ref{Problem.Oleinik} was solved under the two additional assumptions:
a1) $\operatorname{X}$ is a domain in some Carnot group; a2) $u$ is a \textit{homeomorphism.}

The \textit{third step} towards a solution to Problem \ref{Problem.Oleinik} has been taken very recently by S.~V.~Pavlov \cite{Pav}. If we neglect
some technical details, then the essence of his main results can be summarized as follows.
\begin{Prop}
\label{Prop.Pavlov_2}
Let $\operatorname{Y}=(\operatorname{Y},\operatorname{d},\underline{y})$ be an arbitrary pointed metric space. Let $\Omega \subset \mathbb{R}^{2}$ be an arbitrary open nonempty set.
Given $p \in [1,\infty]$,
a map $ u:\Omega \to \operatorname{Y}$ belongs to the Ambrosio--Reshetnyak--Sobolev space $W_{p}^{1}(\Omega,\operatorname{Y})$
if and only if, for each $1$-Lipschitz function $h:\operatorname{Y} \to \mathbb{R}$ with $h(\underline{y})=0$, the composition $h \circ  u$ belongs to $W_{p}^{1}(\Omega)$.

If, in addition $\mathcal{L}^{2}(\Omega) < +\infty$, then, for each $\epsilon > 0$ and each $c > 12$,
there is a compact $K_{\epsilon} \subset \Omega$ and a 1-Lipschitz map $\underline{h}:\operatorname{Y} \to \mathbb{R}^{2}$
such that $\mathcal{L}^{2}(\Omega \setminus K_{\epsilon,c}) < \epsilon$ and
\begin{equation}
\label{eqq.Pavlov_estimate}
|D u|(x) \le c |D(\underline{h}\circ u)|(x) \quad \hbox{for $\mu$-a.e. $x \in K_{\epsilon,c}$}.
\end{equation}
\end{Prop}
Note that Pavlov's result is highly elegant, but the corresponding approach is essentially based on the two dimensional Euclidean structure of the source space.
In particular, it relies on delicate Sard-type arguments
and reductions to one-dimensional techniques. Such tools are completely
inaccessible not only in the case of general metric measure spaces as source spaces but even for $\Omega \subset \mathbb{R}^{n}$ with $n \geq 3$.




\textit{Finally}, after the draft version of the main building blocks of the present paper has been formed, R.~D.~Oleinik, who independently nurtured ideas around Problem \ref{Problem.Oleinik} since 2024,
had a flash of insight that he is also close to a substantial progress. More precisely, a combination of his very recent results \cite{Oleinik1} with
an appropriate finite-dimensional reduction leads to solution of the Problem \ref{Problem.Oleinik} under certain mild assumptions on $\operatorname{X}$ and $\operatorname{Y}$.
I was immediately informed about that and after a short discussion encouraged him to write an independent paper \cite{Oleinik2}.

\subsection{Main results and ideas}
We briefly fix some notation and terminology. The details can be found in Section 2.
Given $p \in [1,\infty)$, we say that a m.m.s. $\operatorname{X}=(\operatorname{X},\rho,\mu)$
is $p$-admissible and write $\operatorname{X} \in \mathcal{A}_{p}$ if it satisfies the following assumptions:

\begin{itemize}
\item[\((\operatorname{A1})\)] the metric space $(\operatorname{X},\rho)$ is complete;

\item[\((\operatorname{A2})\)] the measure $\mu$ is uniformly locally doubling;

\item[\((\operatorname{A3})\)] the triple $(\operatorname{X},\rho,\mu)$ admits a weak  local
$(1,p)$-Poincar\'e inequality.
\end{itemize}

Given $q \in (1,\infty)$ and $r > 0$, for each $f \in L_{q}(\operatorname{X})$ by $\operatorname{M}^{\le r}_{q}[f]$ we denote the \textit{local
maximal function} of $f$ (see Section 3 for details).

Our \textit{first main result} which can be of independent interest reads as follows.
\begin{Th}
\label{Th.main_first}
For each $p \in (1,\infty)$ and each $\operatorname{X}=(\operatorname{X},\rho,\mu) \in \mathcal{A}_{p}$, there exist constants $\operatorname{C}_{1},\operatorname{C}_{2} > 0$ and a parameter $q \in (1,p)$ such that the following holds.
Given a complete pointed metric space $\operatorname{Y}=(\operatorname{Y},\operatorname{d},\underline{y})$, for each map $u \in W_{p}^{1}(\operatorname{X},\operatorname{Y})$ and each parameter $r \in (0,1)$, there exist
a set $\mathcal{S}_{u,r} \subset \operatorname{X}$ and a $1$-Lipschitz function $\underline{h}_{u,r}$ with $\underline{h}_{u,r}(\underline{y})=0$ such that:

\begin{itemize}
\item[\((1)\)] the following inequality holds
\begin{equation}
\label{eqq.main_set}
\int\limits_{\operatorname{X}}(|Du|_{p}(x))^{p}\,d\mu(x) \le \operatorname{C}_{1}\int\limits_{\mathcal{S}_{u,r}}(|Du|_{p}(x))^{p}\,d\mu(x);
\end{equation}

\item[\((2)\)] the following inequality holds
\begin{equation}
\label{eqq.main_Lip}
|Du|_{p}(x) \le \operatorname{C}_{2}\operatorname{M}^{\le r}_{q}[|D(\underline{h}_{u,r}\circ u)|_{p}](x) \quad \hbox{for $\mu$-a.e. $x \in \mathcal{S}_{u,r}$}.
\end{equation}
\end{itemize}
\end{Th}

The remarkable fact about Theorem \ref{Th.main_first} is that the constants $\operatorname{C}_{1},\operatorname{C}_{2} > 0$ do not depend on $\operatorname{Y}$.
This fact combined with smart approximations of a given map $u \in L_{p}(\operatorname{X},\operatorname{Y}) \setminus W_{p}^{1}(\operatorname{X},\operatorname{Y})$
by maps $u_{n} \in W_{p}^{1}(\operatorname{X},\operatorname{Y})$, $n \in \mathbb{N}$, with finite-dimensional images and big ``energies'' lead to the \textit{second main result} of this paper.
\begin{Th}
\label{Th.main_second}
Let $p \in (1,\infty)$, $\operatorname{X}=(\operatorname{X},\rho,\mu) \in \mathcal{A}_{p}$, and let $\operatorname{Y}=(\operatorname{Y},\operatorname{d},\underline{y})$ be a complete pointed metric space.
Given a Borel map $\underline{u}:\operatorname{X} \to \operatorname{Y}$ with a separable image, an equivalence class $u$ of $\underline{u}$ belongs to the Ambrosio--Reshetnyak--Sobolev class $W_{p}^{1}(\operatorname{X},\operatorname{Y})$ if and only if, for each $1$-Lipschitz function $h:\operatorname{Y} \to \mathbb{R}$ with $h(\underline{y})=0$, the composition $h \circ  u$ belongs to the Sobolev space $W_{p}^{1}(\operatorname{X})$.
\end{Th}

For the reader convenience, below we informally sketch main ideas needed for the proofs of Theorems \ref{Th.main_first} and \ref{Th.main_second}.
We would like to underline that our methods are new even in the case $\operatorname{X}=(\mathbb{R}^{n},\mathcal{L}^{n},\|\cdot\|)$ and
completely differ from those developed in \cite{OT,Oleinik2,PavVod,Pav}.

The proof of Theorem \ref{Th.main_first} is based on the following ideas:

\begin{itemize}
\item[\(($I.1.1$)\)] Under the assumptions of the theorem, each map $u \in W_{p}^{1}(\operatorname{X},\operatorname{Y})$
has the Lusin--Lipschitz property. This fact in combination with locality properties of energies of maps allow to transfer some infinitesimal arguments from Lipschitz to Sobolev context.

\item[\(($I.1.2$)\)] Under the assumptions of the theorem and separability of $\operatorname{Y}$,
one can show that each Lipschitz map $u:\operatorname{X} \to \operatorname{Y}$
has the $\operatorname{Lip}-\operatorname{lip}$ property. Furthermore, the lower and the upper pointwise Korevaar--Sch\"oen energies appear to be comparable.

\item[\(($I.1.3$)\)] The third crucial idea is that the push-forward measure $u_{\sharp}((|Du|_{p})^{p}\mu)$ is uniformly asymptotically doubling
on the image of the set, where $|Du|_{p}$ is nonzero.

\item[\(($I.1.4$)\)] Having at our disposal $(\operatorname{I}.1.2)$ and $(\operatorname{I}.1.3)$, by the standard $5B$-Vitali lemma we can
build a special family of balls $\{B_{i}\}$ in $\operatorname{Y}$ whose pre-images under the map $u$ ``bite'' big pieces of the ``Lebesgue
layers'' of $|Du|_{p}$.
\end{itemize}

The proof of Theorem \ref{Th.main_second} is based on the following ideas:

\begin{itemize}
\item[\(($I.2.1$)\)] We reduce the general case to the case of mappings with compact images.

\item[\(($I.2.2$)\)] Using a modified Fr\'echet-type imbedding we reduce the case of maps with compact images to the case of maps
with finite-dimensional images.

\item[\(($I.2.3$)\)] We use the closed graph theorem in combination with the independence of constants $\operatorname{C}_{1}$, $\operatorname{C}_{2}$ in Theorem \ref{Th.main_first}
on $\operatorname{Y}$.
\end{itemize}

We should draw the reader attention that, while the idea of fruitfulness of finite dimensional reductions for solving Problem \ref{Problem.Oleinik}
has been observed by S.~V.~Pavlov in \cite{Pav}, his
method of the corresponding reduction is completely different from our ideas $(\operatorname{I}.2.1),(\operatorname{I}.2.2)$. Indeed, the roots of Pavlov's approach go back to the Reshetnyak o-boundedness concept \cite{Re}, whereas
our approach is based on a new localized version of the Fr\'echet imbedding in combination with some weak convergence ideas used in \cite{GT}.

\subsection{Open problems}

A few words should be said about perspectives for the future investigations and some open problems.

\begin{itemize}
\item[\(($P1$)\)] We do not study Problem \ref{Problem.Oleinik} in the case $p=1$ because it
is quite specific and calls for completely different methods and techniques.

\item[\(($P2$)\)] We do not consider
analogs of Problem \ref{Problem.Oleinik} in the $BV$-case. As it was said above, even in the simplest one-dimensional case
there is no chance to control jumps of maps by Lipschitz functions. Under the additional continuity assumptions we still need
some new methods and tools to obtain any reasonable analog of Proposition \ref{Prop.OT_1}.


\item[\(($P3$)\)] Is it possible to establish analogs of Theorems \ref{Th.main_first}, \ref{Th.main_second} without \textit{any
additional} measure theoretic and metric assumptions on $\operatorname{X}$ and $\operatorname{Y}$, respectively?

\end{itemize}

As was mentioned above, independently and almost simultaneously R.~D.~Oleinik also solved Problem \ref{Problem.Oleinik} in \cite{Oleinik2}.
More precisely, using abstract techniques he proved an analog of Theorem \ref{Th.main_second} under less restrictive assumptions on $\operatorname{X}$.
Without any additional assumptions on $\operatorname{Y}$, his approach covers the case when $\operatorname{X}$ has a finite Hausdorff dimension and
functions in the Sobolev space $W_{p}^{1}(\operatorname{X})$ have the Lusin--Lipschitz property. However, any analogs of Theorem \ref{Th.main_first} have
not been considered in \cite{Oleinik2} under such general assumptions on $\operatorname{X}$, and we do not know, whether methods developed in \cite{Oleinik1,Oleinik2} are suitable for that.

\section{Preliminaries}

The aim of this section is to fix the main terminology and notation that will be used throughout the paper.
Some important technical propositions and lemmas will be considered in the next paragraph.
This is done for readers convenience. The expert reader can briefly catch notation and skip the next section taking his time, whereas the non-expert reader
can go line by line through the current and the next sections.

\subsection{Metric spaces}
Given a metric space $(\operatorname{Z},\operatorname{d}_{\operatorname{Z}})$, the symbol $B$ always stands for an \textit{open ball} in $\operatorname{Z}$ with
a prescribed center and radius. If $B$ is a ball in $\operatorname{Z}$, then we let $\overline{B}$ denote the \textit{closed ball} with the same center and radius.
More precisely, for some $z \in \operatorname{Z}$ and $r > 0$,
\begin{equation}
\label{eqq.balls}
B=B_{r}(z):=\{z' \in \operatorname{Z}: \operatorname{d}(z',z) < r\}, \quad \overline{B}=\overline{B}_{r}(z):=\{z' \in \operatorname{Z}: \operatorname{d}(z,z') \le r\}.
\end{equation}
Furthermore, given a ball $B=B_{r}(z)$ in $\operatorname{Z}$ we let $r(B)$ denote its radius, $r(B):=r$.
Finally, if $c > 0$, then we will use notation $cB:=B_{rc}(z)$.

Given a metric space $(\operatorname{Z},\operatorname{d}_{\operatorname{Z}})$ and a set $E \subset \operatorname{Z}$, by $\chi_{E}$ we
denote the \textit{characteristic function} of $E$. If $E$ is nonempty,
we will occasionally consider $E$ as a \textit{metric subspace} of $\operatorname{Z}$, i.e. as a metric space
whose metric $\operatorname{d}_{\operatorname{Z}}|_{E}$ is the restriction of $\operatorname{d}_{\operatorname{Z}}$ to $E$. The \textit{diameter} of a
set $E \subset \operatorname{Z}$ is defined by the equality
$\operatorname{diam}E:=\sup\{\operatorname{d}(z_{1},z_{2}):z_{1},z_{2} \in E\}$, as usual. Furthermore, given two nonempty sets $E_{1},E_{2} \subset \operatorname{Z}$ we define the \textit{distance}
between them by the equality $\operatorname{dist}(E_{1},E_{2}):=\inf\{\operatorname{d}_{Z}(z_{1},z_{2}):z_{1} \in E_{1}, z_{2} \in E_{2}\}$.

Throughout this paper, usually, we work with the so-called \textit{pointed metric spaces} considered as triples $\operatorname{Z}=(\operatorname{Z},\operatorname{d}_{\operatorname{Z}},\underline{z})$ composed
of a metric space  $(\operatorname{Z},\operatorname{d}_{\operatorname{Z}})$ and a certain distinguished point $\underline{z} \in \operatorname{Z}$. This is extremely convenient for the work with metric-valued
maps.

Given $n \in \mathbb{N}$, the symbol $\mathbb{R}^{n}$ means the usual $n$-dimensional real \textit{linear space}. We
equip $\mathbb{R}^{n}$ with the $\sup$-norm $\|\cdot\|_{\infty}$, i.e.
$\|x\|_{\infty}:=\max\{|x_{1}|,...,|x_{n}|\}$ unless otherwise is stated. We set $\mathbb{R}^{n}_{\infty}:=(\mathbb{R}^{n},\|\cdot\|_{\infty})$ for brevity.
The symbol $l_{\infty}$ denotes the Banach space of all bounded sequences $x=(x_{1},...,x_{i},....)$ of real numbers equipped with the
norm $\|x\|_{\infty}:=\sup_{i \in \mathbb{N}}|x_{i}|$.

We recall the famous Fr\'echet imbedding (see Section 4.1 in \cite{HKST} for details).
\begin{Lm}
\label{Lm.Frechet}
If $\operatorname{Z}=(\operatorname{Z},\operatorname{d},\underline{z})$ is a pointed separable metric space, then the map
$\iota:\operatorname{Z} \to l_{\infty}$ given by
\begin{equation}
\label{eqq.Frechet}
\iota(z):=(\operatorname{d}(z,z_{1})-\operatorname{d}(z_{1},\underline{z}),...,\operatorname{d}(z,z_{n})-\operatorname{d}(z_{n},\underline{z}),...), \quad z \in \operatorname{Z}
\end{equation}
is an isometry of $\operatorname{Z}$ into $l_{\infty}$ sending $\underline{z}$ to $0$.
\end{Lm}

\subsection{Measures}
Let $\operatorname{Z} = (\operatorname{Z},\operatorname{d}_{\operatorname{Z}})$ be a metric space. Following traditions adopted in the contemporary measure theory (see Ch.3 in \cite{HKST} for details), by a \textit{measure} on $\operatorname{Z}$ we always mean an \textit{outer measure}, i.e. a countably subadditive set-function  $\mathfrak{m}:2^{\operatorname{Z}} \to [0,+\infty]$ such that $\mathfrak{m}(\emptyset)=0$.
We say that a measure $\mathfrak{m}$ on $\operatorname{Z}$ is \textit{Borel} if each Borel subset of $\operatorname{Z}$ is $\mathfrak{m}$-measurable.
Given a measure $\mathfrak{m}$ on $\operatorname{Z}$ and a nonempty set $E \subset \operatorname{Z}$,
we denote by $\mathfrak{m}\lfloor_{E}$ the \textit{restriction} of $\mathfrak{m}$ to $E$, i.e. the measure defined by the equality
\begin{equation}
\label{eqq.restriction_measure}
\mathfrak{m}\lfloor_{E}(E'):=\mathfrak{m}(E \cap E') \quad \hbox{for each $E' \subset \operatorname{Z}$}.
\end{equation}
If $\mathfrak{m}$ is a Borel measure on $\operatorname{X}$, then, for each $E \subset \operatorname{X}$, the measure $\mathfrak{m}\lfloor_{E}$ is
Borel (see Lemma 3.3.13 in \cite{HKST} for details).

Let $\operatorname{X}=(\operatorname{X},\rho)$ and
$\operatorname{Y}=(\operatorname{Y},\operatorname{d})$ be metric spaces. Given a measure $\mathfrak{m}$ on $\operatorname{X}$ and a map $F: \operatorname{X} \to \operatorname{Y}$,
the \textit{push-forward measure} $F_{\sharp}\mathfrak{m}$ is the measure on $\operatorname{Y}$ defined by the equality
\begin{equation}
\label{eqq.push_forward}
(F_{\sharp}\mathfrak{m})(E):=\mathfrak{m}(F^{-1}(E)) \quad \hbox{for each set $E \subset \operatorname{Y}$}.
\end{equation}

Throughout the paper by a \textit{metric measure space}, m.m.s. for short, we always mean a triple $(\operatorname{X},\operatorname{d},\mathfrak{m})$, where
$(\operatorname{X},\operatorname{d})$ is a separable metric space and $\mathfrak{m}$ is a nontrivial Borel regular measure on $(\operatorname{X},\operatorname{d})$
giving finite mass to bounded sets.

\subsection{Borel and Lipschitz maps} Given abstract nonempty
sets $\operatorname{X},\operatorname{Y}$, a map $F:\operatorname{X} \to \operatorname{Y}$, and a nonempty set $E \subset \operatorname{X}$, by $F|_{E}$ we denote the \textit{pointwise restriction} of $F$ to $E$.

Given metric spaces $\operatorname{X}=(\operatorname{X},\rho)$ and $\operatorname{Y}=(\operatorname{Y},\operatorname{d})$, we let $\mathfrak{B}(\operatorname{X},\operatorname{Y})$ denote the
set of all \textit{Borel maps} $F$ from $\operatorname{X}$ to $\operatorname{Y}$. We also put $\mathfrak{B}(\operatorname{X}):=\mathfrak{B}(\operatorname{X},\mathbb{R})$.

Given a m.m.s. $\operatorname{X}:=(\operatorname{X},\rho,\mu)$ and a pointed complete metric space $\operatorname{Y}:=(\operatorname{Y},\operatorname{d},\underline{y})$ we denote by $L_{0}(\operatorname{X},\operatorname{Y})$ the
collection of all \textit{equivalence classes} (up to $\mu-$a.e. equality) of maps from $\mathfrak{B}(\operatorname{X},\operatorname{Y})$ with \textit{separable}
range. Given $p \in [1,+\infty]$, we put
\begin{equation}
\label{eqq.Lebesgue_space}
L_{p}(\operatorname{X},\operatorname{Y}):=
\begin{cases}
\{u \in L_{0}(\operatorname{X},\operatorname{Y}): \int_{\operatorname{X}}\operatorname{d}^{p}( u(x),\underline{y})\,d\mu(x) < +\infty\}, \quad \hbox{if $p < +\infty$};\\
\{u \in L_{0}(\operatorname{X},\operatorname{Y}): \operatorname{ess}\sup_{x \in \operatorname{X}}\operatorname{d}(u(x),\underline{y}) < +\infty\}, \quad \hbox{if $p = +\infty$}.
\end{cases}
\end{equation}

\textbf{Important notation.} If $\operatorname{Y}$ is the classical real line $\mathbb{R}$ whose distance between two numbers
is given by the modulus of their difference, then we omit it from the corresponding notation, i.e. we always put $L_{p}(\operatorname{X}):=L_{p}(\operatorname{X},\mathbb{R})$, $p \in [1,\infty]$.
Furthermore, given $p \in [1,\infty]$, we write $f \in L_{p}^{\rm loc}(\operatorname{X})$, if $f \in L_{p}(B)$ for each nonempty ball $B$ in $\operatorname{X}$ considered as a metric subspace.
Finally, given $f \in L_{p}^{\rm loc}(\operatorname{X})$ and a \textit{bounded} measurable set $E \subset \operatorname{X}$, we put
\begin{equation}
\label{eqq.averaged_integral}
\fint\limits_{E}f(x)\,d\mathfrak{m}(x):=
\begin{cases}
\frac{1}{\mu(E)}\int\limits_{E}f(x)\,d\mu(x), \quad \mu(E) > 0; \\
0, \quad \mu(E) = 0.
\end{cases}
\end{equation}

\begin{Remark}
\label{Rem.representatives}
Throughout this paper, unless otherwise is stated, the words ``a map $ u \in L_{p}(\operatorname{X},\operatorname{Y})$''
mean that, the corresponding equivalence class of $ u$ belongs to $L_{p}(\operatorname{X},\operatorname{Y})$. All equalities and inequalities involving $ u$
are always implied to hold almost everywhere for an arbitrary chosen representative of $ u$. Finally, if $\operatorname{Z}:=(\operatorname{Z},\operatorname{d}_{\operatorname{Z}})$ is a metric space, $F \in \mathfrak{B}(\operatorname{Y},\operatorname{Z})$ and
$ u \in L_{0}(\operatorname{X},\operatorname{Y})$, then the symbol $F \circ  u$ means an equivalence class (modulo coincidence $\mu$-a.e.) of the map $F \circ \overline{ u}$ for an arbitrary
chosen Borel representative $\overline{ u}$ of $ u$.
\end{Remark}

Given $L > 0$ and metric spaces $\operatorname{X}=(\operatorname{X},\rho)$, $\operatorname{Y}=(\operatorname{Y},\operatorname{d})$, a map $F:\operatorname{X} \to \operatorname{Y}$ is said to be \textit{$L$-Lipschitz} if
\begin{equation}
\label{eqq.Lipschitz_map}
\operatorname{d}(F(x),F(y)) \le L \rho(x,y) \quad \hbox{for all} \quad x,y \in \operatorname{X}.
\end{equation}
We say that a map $F:\operatorname{X} \to \operatorname{Y}$ is \textit{Lipschitz} if it is $L$-Lipschitz for some $L > 0$.
Furthermore, by $L_{F}$ we denote the global Lipschitz constant of $F$, i.e. the least $L$ for which \eqref{eqq.Lipschitz_map} holds. Given $L > 0$, the symbol
$\operatorname{LIP}_{L}(\operatorname{X},\operatorname{Y})$ denotes the set of all $L$-Lipschitz maps from $\operatorname{X}$ to $\operatorname{Y}$. Finally,
we put $\operatorname{LIP}(\operatorname{X},\operatorname{Y}):=\cup_{L > 0}\operatorname{LIP}_{L}(\operatorname{X},\operatorname{Y})$.

We recall a slightly generalized version of the concepts introduced in \cite{Keith} (see also \cite{KlMac} for a simplified exposition).
Given $F \in \operatorname{LIP}(\operatorname{X},\operatorname{Y})$, for each $r > 0$ and $\underline{x} \in \operatorname{X}$,
the \textit{variation} of $F$ on a ball $B_{r}(\underline{x})$ in $\operatorname{X}$ is defined by the equality
\begin{equation}
\label{eqq.variation_on_a_ball}
\operatorname{var}_{r}[F](\underline{x}):=\sup\Bigl\{\frac{\operatorname{d}(F(x_{1}),F(x_{2}))}{r}:x_{1},x_{2} \in B_{r}(\underline{x})\Bigr\}.
\end{equation}
Sometimes, we will also deal with a little bit different concept of variation. Namely, we put
\begin{equation}
\label{eqq.variation_on_a_ball'}
\widetilde{\operatorname{var}}_{r}[F](\underline{x}):=\sup\Bigl\{\frac{\operatorname{d}(F(\underline{x}),F(x))}{r}:x \in B_{r}(\underline{x})\Bigr\}.
\end{equation}

\begin{Remark}
\label{Rem.different_variations}
By the triangle inequality, we clearly have
\begin{equation}
\label{eqq.different_variations}
\widetilde{\operatorname{var}}_{r}[F](\underline{x}) \le \operatorname{var}_{r}[F](\underline{x}) \le 2\widetilde{\operatorname{var}}_{r}[F](\underline{x}).
\end{equation}
\end{Remark}

Furthermore, given $\underline{x} \in \operatorname{X}$, the \textit{lower} and the \textit{upper pointwise Lipschitz constants} of $F$ at $\underline{x} \in \operatorname{X}$ are defined by the equalities
\begin{equation}
\label{eqq.lower_upper_pointwise_Lipschitz}
\operatorname{lip}[F](\underline{x}):=\varliminf\limits_{r \to 0}\operatorname{var}_{r}[F](\underline{x}), \quad \operatorname{Lip}[F](\underline{x}):=\varlimsup\limits_{r \to 0}\operatorname{var}_{r}[F](\underline{x}).
\end{equation}

\begin{Remark}
\label{Rem.Lip_is_Borel}
The same arguments as in the proof of Lemma 3.1 in \cite{KlMac} allow to conclude that the functions $x \mapsto \operatorname{lip}[F](x)$ and $x \mapsto \operatorname{Lip}[F](x)$ are Borel.
\end{Remark}

We recall the McShane--Whitney extension lemma (see \S 4.1 in \cite{HKST} for the details).
\begin{Lm}
\label{Lm.McShane_Whitney}
Given a nonempty set $E \subset \operatorname{X}$, for each $\widetilde{h} \in \operatorname{LIP}(E,l_{\infty})$,
there is $h \in \operatorname{LIP}(\operatorname{X},l_{\infty})$ such that $h|_{E}=\widetilde{h}$,  $\|\widetilde{h}\|_{L_{\infty}(E)}=\|h\|_{L_{\infty}(\operatorname{X})}$, and
$L_{h}=L_{\widetilde{h}}$.
\end{Lm}

\textbf{Important notation.} Given a pointed metric space $\operatorname{Y}=(\operatorname{Y},\operatorname{d},\underline{y})$ we set
$$
\operatorname{LIP}(\operatorname{Y},\underline{y}):=\{h \in \operatorname{LIP}(\operatorname{Y}):h(\underline{y})=0\}.
$$
\begin{Remark}
\label{Rem.modified_Lipschitz}
Clearly, if we consider $\operatorname{LIP}(\operatorname{Y},\underline{y})$ as a linear space and equip it with the norm
\begin{equation}
\|h\|_{\operatorname{LIP}(\operatorname{Y},\underline{y})}:=L_{h}
\end{equation}
it becomes a Banach space.
\end{Remark}

\subsection{Admissible metric measure spaces}
Unfortunately, generic m.m.s. can be very tricky. In order to build a nice fruitful theory we impose some additional
assumptions that are quite popular in contemporary Geometric Analysis.

We recall that, given a m.m.s. $\operatorname{X}=(\operatorname{X},\rho,\mu)$, the measure $\mu$ is said to be
\textit{uniformly locally doubling} if, $\mu(B) > 0$ for any nonempty ball $B$ in $\operatorname{X}$, and, for each $R > 0$,
\begin{equation}
\label{eqq.doubling_measure}
\operatorname{D}_{\mu}(R):=\sup\limits_{x \in \operatorname{X}}\sup\limits_{r \in (0,R)}\frac{\mu(B_{2r}(x))}{\mu(B_{r}(x))} < +\infty.
\end{equation}

Given $p \in [1,\infty)$, we say that a m.m.s. $\operatorname{X}=(\operatorname{X},\rho,\mu)$ supports a \textit{weak local $(1,p)$-Poincar\'e} inequality if, for each
$R > 0$, there are constants $\operatorname{c}=\operatorname{c}(\operatorname{X},p,R) > 0$ and $\lambda=\lambda(\operatorname{X},p,R) > 0$ such that for each $f \in \operatorname{LIP}(\operatorname{X})$ it holds
\begin{equation}
\label{eqq.Poincare}
\begin{split}
&\fint\limits_{B_{r}(x)}\Bigl|f(y)-\fint\limits_{B_{r}(x)}f(z)\,d\mathfrak{m}(z)\Bigr|\,d\mu(y)\\
&\le \operatorname{c} r\Bigl(\fint\limits_{B_{\lambda r}(x)}\Bigl(\operatorname{lip}[f](y)\Bigr)^{p}\,d\mu(y)\Bigr)^{\frac{1}{p}} \quad \hbox{for all $(x,r) \in \operatorname{X}\times(0,R)$}.
\end{split}
\end{equation}

\begin{Def}
\label{Def.p_admissible}
Given $p \in [1,\infty)$, we say that a m.m.s. $\operatorname{X}=(\operatorname{X},\rho,\mu)$ is $p$-admissible and write $\operatorname{X} \in \mathcal{A}_{p}$ if
the metric space $(\operatorname{X},\rho)$ is complete, the measure $\mu$ is uniformly locally doubling, and the triple $(\operatorname{X},\rho,\mu)$ supports a weak local $(1,p)$-Poincar\'e inequality.
\end{Def}

We recall that a metric measure space with globally doubling measure supporting global weak $(1,p)$-Poincar\'e inequality in fact support global weak $(1,q)$-Poincar\'e inequality
for some $q \in [1,p)$. This is the famous result due to Keith and Zhong (see, for example, Ch.12 in \cite{HKST} for the details). Fortunately, there exists a local
analog of this beautiful self-improvement property established in \cite{EB}.
\begin{Prop}
\label{Prop.Poincare_self_improvement}
Let $\operatorname{X}=(\operatorname{X},\rho,\mu) \in \mathcal{A}_{p}$ for some $p \in (1,\infty)$. Then, for each $R > 0$, there exists $q(R) \in (1,p)$ such that
\eqref{eqq.Poincare} holds true with $p$ replaced by $q(R)$, with $\operatorname{c}$ replaced by some $\widetilde{\operatorname{c}} \geq \operatorname{c}$,
and with $\lambda$ replaced by some $\widetilde{\lambda} \geq \lambda$.
\end{Prop}

We recall that, given $R > 0$, a metric space $(\operatorname{Z},\operatorname{d}_{\operatorname{Z}})$ is said to be \textit{geodesic at the scale $R$} if, for each pair of points $z_{1},z_{2} \in \operatorname{Z}$ with $\operatorname{d}_{\operatorname{Z}}(z_{1},z_{2}) \le R$ there is a curve joining $z_{1}$ and $z_{2}$ whose length is equal to the distance $\operatorname{d}_{\operatorname{Z}}(z_{1},z_{2})$.

Recall also that two metrics $\operatorname{d}_{1}$ and $\operatorname{d}_{2}$ in a set $\operatorname{Z}$ are said to be \textit{biLipschitz equivalent}
if there is a constant $C \geq 1$ such that
\begin{equation}
\label{eqq.bilipschitz_metric}
\frac{1}{C}\operatorname{d}_{2}(z_{1},z_{2}) \le \operatorname{d}_{1}(z_{1},z_{2}) \le C \operatorname{d}_{2}(z_{1},z_{2}) \quad \hbox{for every pair of points $z_{1},z_{2} \in \operatorname{Z}$}.
\end{equation}
Arguing as in the Corollary 8.3.16 and Lemma 8.3.18 from \cite{HKST} we get the following assertion.
\begin{Prop}
\label{Prop.Poincare_forgotten}
Let $p \in (1,\infty)$ and $\operatorname{X}=(\operatorname{X},\rho,\mu) \in \mathcal{A}_{p}$. Then, for each $R > 0$, there exists a geodesic at the scale $R > 0$ metric $\widetilde{\rho}$ on $\operatorname{X}$ such that
$\rho$ and $\widetilde{\rho}$ are biLipscgitz equivalent in $\operatorname{X}$ and $\widetilde{\operatorname{X}}=(\operatorname{X},\widetilde{\rho},\mu) \in \mathcal{A}_{p}$.
\end{Prop}

\subsection{Sobolev functions}
We recall basic terminology concerning modern theory of scalar-valued Sobolev functions following approach given in \cite{GT} (see also \cite{AGS,AmbPas} for the comprehensive treatment of the subject and equivalent approaches).

We fix a m.m.s. $\operatorname{X}=(\operatorname{X},\rho,\mu)$ and $p \in (1,\infty)$. Assume that $(\operatorname{X},\rho)$ is \textit{complete and separable}.
Let $p'$ be the conjugate H\"older exponent, i.e. $\frac{1}{p}+\frac{1}{p'}=1$.
We shall denote by $C([0,1], \operatorname{X}])$ the (complete and separable) space of continuous curves in $\operatorname{X}$ defined on $[0,1]$
equipped with the ``sup''-distance. Given $t \in [0,1]$, the evaluation map $e_{t}: C([0,1],\operatorname{X}) \to \operatorname{X}$ is defined
by $e_{t}(\gamma):=\gamma(t)$, $\gamma \in C([0,1],\operatorname{X})$.
A curve $\gamma \in C([0,1],\operatorname{X})$ is said to be \textit{absolutely continuous} provided there is $f \in L_{1}([0,1])$ such that
\begin{equation}
\notag
\rho(\gamma(t_{1}),\gamma(t_{2})) \le \int\limits_{t_{1}}^{t_{2}}f(s)\,ds \quad \hbox{for all $0 \le t_{1} \le t_{2} \le 1$.}
\end{equation}
The least - in the $\mathcal{L}^{1}$-a.e. sense - such $f$
is called the \textit{metric speed of $\gamma$} and denoted by $|\dot{\gamma}|$. In what
follows, we put $\int_{0}^{1}|\dot{\gamma}|(t)\,dt = +\infty$ if $\gamma$ is not absolutely continuous.

Given $q \in (1,\infty)$, a Borel probability measure $\pi$ on $C([0,1],\operatorname{X})$ is said to be a \textit{$q$-test plan} on $\operatorname{X}$ if there is a constant $C > 0$
such that $(e_{t})_{\sharp}\pi \le C \mu$ for all $t \in [0,1]$ and
\begin{equation}
\notag
\int\int\limits_{0}^{1}|\dot{\gamma}|^{q}(t)\,dt\,d\pi(\gamma) < +\infty.
\end{equation}

\begin{Def}
\label{Def.Sobolev_function}
Given $p \in (1,\infty)$, we say that $f \in L_{p}(\operatorname{X})$ belongs to the Sobolev space $W_{p}^{1}(\operatorname{X})$ provided there exists
a nonnegative function $G \in L_{p}(\operatorname{X})$ such that
\begin{equation}
\notag
\int|f(\gamma(1))-f(\gamma(0))|\,d\pi(\gamma) \le \int\int\limits_{0}^{1}G(\gamma(t))|\dot{\gamma}|(t)\,dt\,d\pi(\gamma) \quad \hbox{for each $p'$-test plan $\pi$}.
\end{equation}
Any such $G$ is called $p$-weak upper gradient of $f$.
\end{Def}

\begin{Remark}
\label{Rem.dependence_on_p}
It is easy to show that there exists a \textit{minimal} (in the $\mu$-a.e. sense) \textit{$p$-weak upper gradient} $G$ of $f$ (see \cite{AGS} for details), which will be denoted by $|Df|_{p}$.
The index $p$ in the notation $|Df|_{p}$ stresses the dependence of the minimal $p$-weak upper gradient of $f$ on $p$ (see \cite{DiMarinoSpeight15} for details).
\end{Remark}

\subsection{Sobolev maps}
In this section we briefly recall several approaches to Sobolev metric-valued maps following the terminology introduced in our
joint paper with N.~Gigli \cite{GT}.
Unless otherwise is stated, throughout the section we assume that $\operatorname{X}=(\operatorname{X},\rho,\mu)$ is a complete separable m.m.s. and
$\operatorname{Y}=(\operatorname{Y},\operatorname{d},\underline{y})$ is a pointed complete space.

We start with the \textit{Ambrosio--Reshetnyak approach} to Sobolev mappings.
\begin{Def}
\label{Def.Resh_Sob}
Given $p \in (1,\infty)$, an element $u \in L_{p}(\operatorname{X},\operatorname{Y})$ belongs to the Ambrosio--Reshetnyak--Sobolev class
$W_{p}^{1}(\operatorname{X},\operatorname{Y})$ provided there is a nonnegative $G \in L_{p}(\operatorname{X})$
such that for every $h \in \operatorname{LIP}_{1}(\operatorname{Y},\underline{y})$ it holds $h \circ  u \in W_{p}^{1}(\operatorname{X})$ and
\begin{equation}
\label{eqq.Resh_Sob}
|D(h \circ u)|_{p}(x) \le G(x) \quad \hbox{for $\mu$-a.e. $x \in \operatorname{X}$}.
\end{equation}
Any such $G$ is called $p$-weak upper gradient of $u$.
\end{Def}

\begin{Remark}
\label{Rem.dependence_on_p'}
It is easy to see that there exists a minimal (in the $\mu$-a.e. sense) \textit{$p$-weak upper gradient} $G$ of $u$, which will be denoted by $|Du|_{p}$.
In fact, $|Du|_{p}$ is nothing but the essential supremum of $|D(h\circ u)|_{p}$ as $h$ varies among all $h \in \operatorname{LIP}_{1}(\operatorname{Y},\underline{y})$. The presence of
index $p$ in our notation is justified by the same reason as in the scalar-valued case (see Remark \ref{Rem.dependence_on_p}). One can show though that in the case when
$\operatorname{X}$ is an open set in $\mathbb{R}^{n}$ (equipped with the Euclidean distance and the corresponding Lebesgue measure $\mathcal{L}^{n}$), for each
$p,q \in (1,\infty)$ and any $u \in W^{1}_{p}(\operatorname{X},\operatorname{Y}) \cap W^{1}_{q}(\operatorname{X},\operatorname{Y})$ we have $|Du|_{p} = |Du|_{q}$ in the $\mathcal{L}^{n}$-a.e. sense.
\end{Remark}

\begin{Remark}
\label{Rem.Sobolev_forgotten}
It is easy to see that if $\widetilde{\operatorname{X}}=(\operatorname{X},\widetilde{\rho},\mu)$ and the metrics $\rho,\widetilde{\rho}$ are biLipschitz equivalent on $\operatorname{X}$, then
$W_{p}^{1}(\operatorname{X},\operatorname{Y}) = W_{p}^{1}(\widetilde{\operatorname{X}},\operatorname{Y})$ as sets. Furthermore, if $\operatorname{I}:\operatorname{X} \to \widetilde{\operatorname{X}}$ is the
identical map, then for each $u \in W_{p}^{1}(\operatorname{X},\operatorname{Y})$ we have $\widetilde{u}:=u \circ \operatorname{I}^{-1} \in W_{p}^{1}(\widetilde{\operatorname{X}},\operatorname{Y})$ and
\begin{equation}
\label{eqq.Sobolev_bilipschitz}
\frac{1}{C}|D\widetilde{u}|_{p}(x) \le |Du|_{p}(x) \le C |D\widetilde{u}|_{p}(x) \quad \hbox{for $\mu$-a.e. $x \in \operatorname{X}$},
\end{equation}
where the constant $C > 0$ in \eqref{eqq.Sobolev_bilipschitz} is the same as in \eqref{eqq.bilipschitz_metric}.
\end{Remark}

We continue with the \textit{Hajlasz approach} to Sobolev maps.
\begin{Def}
\label{Def.Hajlasz_Sobolev}
Given $p \in (1,\infty)$, an element $u \in L_{p}(\operatorname{X},\operatorname{Y})$ belongs to the Hajlasz--Sobolev class
$\operatorname{HS}^{1}_{p}(\operatorname{X},\operatorname{Y})$ if for each $R > 0$ there are a nonnegative $G_{R}[u] \in L_{p}(\operatorname{X})$ and a set
$N_{R}[u] \subset \operatorname{X}$ such that $\mu(N_{R}[u]) = 0$ and
\begin{equation}
\begin{split}
\label{eqq.Hajlasz_Sobolev}
&\operatorname{d}(u(x_{1}),u(x_{2})) \\
&\le \rho(x_{1},x_{2})\Bigl(G_{R}[u](x_{1})+G_{R}[u](x_{2})\Bigr) \quad \hbox{for all $x_{1},x_{2} \in \operatorname{X} \setminus N_{R}[u]$ with $\rho(x_{1},x_{2}) \le R$}.
\end{split}
\end{equation}
Any such $G_{R}[u] $ is called Hajlasz gradient
of $u$ at scale $R$.

Furthermore, we say that $u \in L_{p}(\operatorname{X},\operatorname{Y})$ belongs to the local Hajlasz--Sobolev class
$\operatorname{HS}^{1, \rm loc}_{p}(\operatorname{X},\operatorname{Y})$ if there exist some $\underline{R} > 0$, $G_{\underline{R}}[u] \in L_{p}(\operatorname{X})$ and
$N_{\underline{R}}[u] \subset \operatorname{X}$ with $\mu(N_{\underline{R}}[u])=0$ such that inequality \eqref{eqq.Hajlasz_Sobolev}
holds true with $R$ replaced by $\underline{R}$.
\end{Def}

To describe the third approach to Sobolev maps we should recall the concept of the \textit{energy} of a $p$-integrable map.
\begin{Def}
\label{Def.energy}
Let $q \in [1,+\infty]$ and $u \in L_{q}^{\rm loc}(\operatorname{X},\operatorname{Y})$. Given $r > 0$ and $\underline{x} \in \operatorname{X}$, the $q$-energy of $u$ on a ball $B_{r}(\underline{x})$ (in the Korevaar--Sch\"oen sense)
is defined by the equality
\begin{equation}
\operatorname{e}_{q,r}[u](\underline{x}):=
\begin{cases}
\Bigl(\fint\limits_{B_{r}(\underline{x})}\frac{\operatorname{d}^{q}(u(x'), u(\underline{x}))}{r^{q}}\,d\mu(x')\Bigr)^{\frac{1}{p}}, \quad \hbox{if $q < +\infty$};\\
\operatorname{ess}\sup\limits_{x' \in \operatorname{X}}\frac{\operatorname{d}(u(x'), u(\underline{x}))}{r}, \quad \hbox{if $q = +\infty$}.
\end{cases}
\end{equation}
Furthermore, the \textit{lower} and the \textit{upper pointwise $q$-energy} of $u$ at $\underline{x}$ is defined by the equalities
\begin{equation}
\label{eqq.energy_densities}
\underline{\operatorname{e}}_{q}[u](\underline{x}):=\varliminf\limits_{r \to 0}\operatorname{e}_{q,r}[ u](\underline{x}), \quad \overline{\operatorname{e}}_{q}[u](\underline{x}):=\varlimsup\limits_{r \to 0}\operatorname{e}_{q,r}[ u](\underline{x}).
\end{equation}
\end{Def}

\begin{Remark}
\label{Rem.energy_variation}
In the particular case when $p=+\infty$, and the measure $\mu$ gives strictly positive mass to any nonempty ball, given
$u \in \operatorname{LIP}(\operatorname{X},\operatorname{Y})$, for each $\underline{x} \in \operatorname{X}$ and each $r > 0$, the energy $\operatorname{e}_{\infty,r}[u](\underline{x})$ has the same flavour as
the variation of $u$ on the ball $B_{r}(\underline{x})$ defined in \eqref{eqq.variation_on_a_ball'}.
\end{Remark}

\begin{Remark}
Using the same ideas as in \cite{KlMac} it is easy to show that, given $q \in [1,\infty]$ and $u \in L_{q}^{\rm loc}(\operatorname{X},\operatorname{Y})$, the
maps $x \mapsto \underline{\operatorname{e}}_{q}[u](\underline{x})$ and $x \mapsto \overline{\operatorname{e}}_{q}[u](\underline{x})$ (considered as mappings between $\operatorname{X}$ and $[0,+\infty]$) are Borel.
\end{Remark}

Now we are ready to formulate the \textit{Korevaar--Sch\"oen approach} to Sobolev mappings.
\begin{Def}
\label{Def.Korevaar_Schoen}
Let $\operatorname{X}=(\operatorname{X},\rho,\mu)$ be a m.m.s., let $\operatorname{Y}=(\operatorname{Y},\operatorname{d},\underline{y})$ be a pointed complete metric space.
Given $p \in (1,\infty)$, an element $ u \in L_{p}(\operatorname{X},\operatorname{Y})$ belongs to $\operatorname{KS}_{p}^{1}(\operatorname{X},\operatorname{Y})$ provided that
\begin{equation}
\label{eqq.KS_total_energy}
\operatorname{E}_{p}[u]:=\varliminf\limits_{r \to 0}\int\limits_{\operatorname{X}}\bigl(\operatorname{e}_{p,r}[u](x)\bigr)^{p}\,d\mu(x) < +\infty.
\end{equation}
\end{Def}

It is remarkable that different approaches to Sobolev mappings in fact coincide \cite{GT}.
\begin{Prop}
\label{Prop.different_approaches}
Given $p \in (1,\infty)$, for each $\operatorname{X} \in \mathcal{A}_{p}$, it holds
\begin{equation}
\label{eqq.equality_approaches_to_Sobolev}
W_{p}^{1}(\operatorname{X},\operatorname{Y})=\operatorname{KS}_{p}^{1}(\operatorname{X},\operatorname{Y})=\operatorname{HS}_{p}^{1}(\operatorname{X},\operatorname{Y}),
\end{equation}
where the equalities should be interpreted as a coincidence of the sets.
 \end{Prop}

\begin{Remark}
\label{Rem.different_approaches}
Strictly speaking, equality \eqref{eqq.equality_approaches_to_Sobolev} was established in \cite{GT} under the more restrictive assumption $\operatorname{X} \in \mathcal{A}_{1}$.
However, a careful analysis of the proof of Corollary 3.10 in \cite{GT} in combination with Proposition \ref{Prop.Poincare_self_improvement} shows
that our Proposition \ref{Prop.different_approaches} is valid. Furthermore, the equality $\operatorname{HS}_{p}^{1, \rm loc}(\operatorname{X},\operatorname{Y})=\operatorname{HS}_{p}^{1}(\operatorname{X},\operatorname{Y})$ holds.
\end{Remark}

The following assertion is similar to Lemma 5.4 from \cite{Oleinik2} (see also Proposition 4.2 in \cite{PavVod}).
\begin{Lm}
\label{Lm.Oleinik_Pavlov}
Let $u \in L_{p}(\operatorname{X},\operatorname{Y})$ be such that $h \circ u \in W_{p}^{1}(\operatorname{X})$ for each $h \in \operatorname{LIP}_{1}(\operatorname{Y},\underline{y})$. Then,
\begin{equation}
\label{eqq.operator_norm}
\sup\limits_{h \in \operatorname{LIP}_{1}(\operatorname{Y},\underline{y})}\|h \circ u\|_{W_{p}^{1}(\operatorname{X})} < +\infty.
\end{equation}
\end{Lm}

\begin{proof}
Consider a map $A:\operatorname{LIP}(\operatorname{Y},\underline{y}) \to W_{p}^{1}(\operatorname{X})$ defined by the equality
\begin{equation}
\notag
A(h):=h \circ u, \quad h \in \operatorname{LIP}(\operatorname{Y},\underline{y}).
\end{equation}
It is easy to see that $A$ is linear. Since $\operatorname{LIP}(\operatorname{Y},\underline{y})$ and $W_{p}^{1}(\operatorname{X})$ are Banach
spaces, by the closed graph theorem, if we show that the graph of $A$ is closed, then we conclude that $A$ is a bounded operator and
\eqref{eqq.operator_norm} follows immediately. To this aim we fix a sequence $\{h_{n}\} \subset \operatorname{LIP}(\operatorname{Y},\underline{y})$
such that for some $h \in \operatorname{LIP}(\operatorname{Y},\underline{y})$, $f \in W_{p}^{1}(\operatorname{X})$ we have
$h_{n} \to h$, $n \to \infty$ in $\operatorname{LIP}(\operatorname{Y},\underline{y})$ and $h_{n} \circ u \to f$, $n \to \infty$ in $W_{p}^{1}(\operatorname{X})$.
It is clear that $h_{n}(x) \to h(x)$, $n \to \infty$ for all $x \in \operatorname{X}$. On the other hand, there is a strictly increasing sequence $\{n_{j}\} \subset \mathbb{N}$ such that
$h_{n_{j}} \circ u(x) \to f(x)$, $n \to \infty$ for $\mu$-a.e. $x \in \operatorname{X}$. This implies that $f=h \circ u$. This completes the proof.
\end{proof}

Surprisingly, the following two observations will be one of the central technical tools in the sequel.
The first one is a refined version of item $iii)$ in Proposition 2.10 from the paper \cite{GT}. It will be crucial in deriving Theorem \ref{Th.weak_main} from
Theorem \ref{Th.main_first} and Lemma \ref{Lm.Oleinik_Pavlov}.
\begin{Prop}
\label{Prop.weak_convergence}
Given $p \in (1,\infty)$, let $\{ u_{n}\} \subset W^{1}_{p}(\operatorname{X},\operatorname{Y})$ and $ u \in L_{p}(\operatorname{X},\operatorname{Y})$ be such that
$u_{n}(x) \to  u(x)$, $n \to \infty$ for $\mu$-a.e. $x \in \operatorname{X}$.
If $|Du_{n}|_{p} \rightharpoonup G$, $n \to \infty$ in $L_{p}(\operatorname{X})$ to some $G \in L_{p}(\operatorname{X})$, then $u \in W^{1}_{p}(\operatorname{X},\operatorname{Y})$ and
\begin{equation}
\label{eqq.weak_convergence}
|Du|_{p}(x) \le G(x) \quad \hbox{for $\mu$-a.e. $x \in \operatorname{X}$}.
\end{equation}
\end{Prop}

\begin{proof}
We fix an arbitrary $h \in \operatorname{LIP}_{1}(\operatorname{Y},\underline{y})$. It is clear that $h \circ u_{n}(x) \to h \circ u(x)$ , $n \to \infty$ for $\mu$-a.e. $x \in \operatorname{X}$.
Furthermore, for each $n \in \mathbb{N}$, we have $|D(h \circ u_{n})|_{p}(x) \le |Du_{n}|_{p}(x)$ for $\mu$-a.e. $x \in \operatorname{X}$. Hence, the sequence
$\{|D(h \circ u_{n})|_{p}\}$ is bounded in $L_{p}(\operatorname{X})$. Then, there is $g_{h} \in L_{p}(\operatorname{X})$ and a subsequence $\{|D(h \circ u_{n_{j}})|_{p}\}$
weakly converging in $L_{p}(\operatorname{X})$ to $g_{h}$. This gives $g_{h}(x) \le G(x)$ for $\mu$-a.e. $x \in \operatorname{X}$. Hence, using
Theorem 5.3 in \cite{AGS} we derive that $h \circ u \in W_{p}^{1}(\operatorname{X})$ and $|D(h \circ u)|_{p}(x) \le G(x)$ for $\mu$-a.e. $x \in \operatorname{X}$.
This completes the proof.
\end{proof}

The second result is an analog of the Proposition \ref{Prop.weak_convergence} in the context of Hajlasz--Sobolev spaces. It will be crucial in deriving Theorem \ref{Th.main_second} from
Theorem \ref{Th.weak_main}.
\begin{Prop}
\label{Prop.weak_convergence'}
Given $p \in (1,\infty)$, let $\{u_{n}\} \subset L_{p}(\operatorname{X},\operatorname{Y})$ and $u \in L_{p}(\operatorname{X},\operatorname{Y})$ be such that:

\begin{itemize}
\item[\(($1$)\)] $u_{n}(x) \to  u(x)$, $n \to \infty$ for $\mu$-a.e. $x \in \operatorname{X}$;

\item[\(($2$)\)] for some $R > 0$ there are $G_{R} \in L_{p}(\operatorname{X})$, a sequence $\{G_{R,n}\} \subset L_{p}(\operatorname{X})$, an increasing sequence of subsets $\{U_{R,n}\}$ of $\operatorname{X}$, and
a representatives $\underline{u}_{n}$ of $u_{n}$, $n \in \mathbb{N}$,
such that $\mu(\operatorname{X} \setminus \bigcup_{n \in \mathbb{N}} U_{R,n}) = 0$, the sequence $\{G_{R,n}\}$ is bounded in $L_{p}(\operatorname{X})$ and
\begin{equation}
\begin{split}
&\operatorname{d}(\underline{u}_{n}(x_{1}),\underline{u}_{n}(x_{2}))\\
&\le \rho(x_{1},x_{2})(G_{R,n}(x_{1})+G_{R,n}(x_{2})) \quad \hbox{for all $x_{1},x_{2} \in U_{R,n}$ with $\rho(x_{1},x_{2}) \le R$}.
\end{split}
\end{equation}
\end{itemize}
Then, $u$ belongs to $\operatorname{HS}^{1}_{p}(\operatorname{X},\operatorname{Y})$. 
\end{Prop}

\begin{proof}
Since $\{G_{R,n}\} \subset L_{p}(\operatorname{X})$ is bounded in $L_{p}(\operatorname{X})$, passing to a subsequence (not relabeled!) if necessary, 
we may assume that $\{G_{R,n}\}$ converges weakly in $L_{p}(\operatorname{X})$ to some $G_{R}$. 
By Lemma \ref{Lm.Frechet} without loss of generality we may assume that $\operatorname{Y}=l_{\infty}$ and $\underline{y}=0$. By Mazur's theorem
there exists a sequence of convex combinations $\{F_{R,j}\}_{j=1}^{\infty}$ of functions $G_{R,n}$, $n \in \mathbb{N}$ converging to $G_{R}$ strongly. More precisely, there is a strictly increasing sequence $\{N_{j}\} \subset \mathbb{N}$
and for each $j \in \mathbb{N}$ there is $\{\alpha_{j,i}\} \subset [0,+\infty)$ such that:

\begin{itemize}
\item[\(($a$)\)] $\sum_{i=N_{j}+1}^{N_{j+1}}\alpha_{j,i}=1$;

\item[\(($b$)\)] $F_{R,j}:=\sum_{i=N_{j}+1}^{N_{j+1}}\alpha_{j,i}G_{R,i}$;

\item[\(($c$)\)] $F_{R,j} \to G_{R}$, $j \to \infty$ in $L_{p}(\operatorname{X})$.
\end{itemize}
It is clear that, for each $j \in \mathbb{N}$ and all $x_{1},x_{2} \in  U_{R,N_{j}}$ with $\rho(x_{1},x_{2}) \le R$, it holds
\begin{equation}
\label{eqq.225}
\begin{split}
&\Bigl|\sum\limits_{i=N_{j}+1}^{N_{j+1}}\alpha_{j,i}\underline{u}_{j}(x_{1})-\sum\limits_{i=N_{j}+1}^{N_{j+1}}\alpha_{j,i}\underline{u}_{j}(x_{2})\Bigr| \le \rho(x_{1},x_{2})(F_{R,j}(x_{1})+F_{R,j}(x_{2})).
\end{split}
\end{equation}
Since $U_{R,n} \subset U_{R,n+1}$ for all $n \in \mathbb{N}$, passing to an appropriate subsequence in \eqref{eqq.225}, we get
\begin{equation}
\begin{split}
&|\underline{u}(x_{1})-\underline{u}(x_{2})|\\
&\le \rho(x_{1},x_{2})(G_{R}(x_{1})+G_{R}(x_{2})) \quad \hbox{for all $x_{1},x_{2} \in \operatorname{X} \setminus E$ with $\rho(x_{1},x_{2}) \le R$},
\end{split}
\end{equation}
where $\underline{u}$ is a representative of $u$ and $E \supset \operatorname{X} \setminus \bigcup_{n \in \mathbb{N}}U_{R,n}$ is such that $\mu(E)=0$.
It remains to take into account Remark \ref{Rem.different_approaches}.

The proof is complete.
\end{proof}

\section{Main technical tools}
This special section devoted to numerous auxiliary concepts and assertions that will be important in proving main results of this paper.
While the most part of them is rather standard in the contemporary Geometric Analysis, yet there are some delicate observations
whose precise statements are difficult to find in the literature.

\subsection{Maximal functions} In this section we fix a metric measure space $\operatorname{X}=(\operatorname{X},\rho,\mu)$.
We recall the following folklore fact.
\begin{Lm}
\label{Lm.Lebesgue_point}
Let $\operatorname{Y}=(\operatorname{Y},\operatorname{d},\underline{y})$ be a pointed complete space.
Given $p \in [1,\infty)$, if $\mu$ is uniformly locally doubling, then, for each $u \in L_{p}^{\rm loc}(\operatorname{X},\operatorname{Y})$,
\begin{equation}
\label{eqq.Lebesgue}
\lim\limits_{r \to 0}\fint\limits_{B_{r}(x)}\operatorname{d}^{p}(u(x),u(y))\,d\mu(y) = 0 \quad \hbox{for $\mu$-a.e. $x \in \operatorname{X}$}.
\end{equation}
\end{Lm}

\begin{proof}
We present only a sketch of the proof.
If $\operatorname{Y}=(\mathbb{R},|\cdot|,0)$ and $p=1$, then the proof can be found in Section 3.4 of \cite{HKST}. The case $p > 1$ follows easily, if one takes into account the inequality $b^{p}-a^{p} \geq (b-a)^{p}$
for $b \geq a \geq 0$. In the general case one might fix an arbitrary Borel
representative $\underline{u}$ of $u$ with separable range, use Lemma \ref{Lm.Frechet} and apply scalar-valued arguments to functions $\operatorname{d}(u,y_{n})$, $n \in \mathbb{N}$ for a countable
set $\{y_{n}\}_{n=1}^{\infty}$ which is dense in $u(\operatorname{X})$.
\end{proof}

\begin{Remark}
\label{Rem.differentiation_integarl}
It follows immediately from Lemma \ref{Lm.Lebesgue_point} that, given $f \in L_{1}^{\rm loc}(\operatorname{X})$, for $\mu$-a.e. $x \in \operatorname{X}$ there exists
\begin{equation}
\label{eqq.differentiability_of_integral}
\lim\limits_{r \to 0}\fint\limits_{B_{r}(x)}f(y)\,d\mu(y) = f(x).
\end{equation}
In particular, given a measurable set $E \subset \operatorname{X}$, $\mu$-a.e. point $x \in E$ is a density point of $E$ and
\begin{equation}
\label{eqq.differentiability_of_integral'}
\lim\limits_{r \to 0}\frac{1}{\mu(B_{r}(x))}\int\limits_{B_{r}(x)\setminus E}f(y)\,d\mu(y) = 0.
\end{equation}
\end{Remark}

The proof of the following assertion repeats almost verbatim the corresponding arguments from the proof of Lemma 8.1.13 in \cite{HKST}.
We omit elementary modifications.
\begin{Lm}
\label{Lm.volume_decay}
If $\mu$ is uniformly locally doubling, then, given $R > 0$,
there exist constants $\theta > 0$ and $C > 0$ depending only on $\operatorname{D}_{\mu}(2R)$ such that, for each $0 < s \le r < R$, each $x \in \operatorname{X}$, and each $y \in B_{r}(x)$, it holds
\begin{equation}
\label{eqq.volume_decay}
\frac{\mu(B_{s}(y))}{\mu(B_{r}(x))} \geq C \Bigl(\frac{s}{r}\Bigr)^{\theta}.
\end{equation}
By $\underline{\theta}_{\mu}$ we denote the infimum of all $\theta$ for which \eqref{eqq.volume_decay} holds.
\end{Lm}

\begin{Remark}
\label{Rem.volume_decay}
Assume that $\operatorname{X}=(\operatorname{X},\rho,\mu)$ is such that the metric space $(\operatorname{X},\rho)$ is geodesic at the scale $R$ for some $R > 0$
and $\mu$ is uniformly locally doubling. Then, repeating arguments
from the proof of Proposition 11.5.3 in \cite{HKST} we obtain that for each  there are $C > 0$ and $\beta \in (0,1]$ such that
\begin{equation}
\mu(B_{r}(x) \setminus B_{r(1-\epsilon)}(x)) \le C \epsilon^{\beta}\mu(B_{r}(x)) \quad \hbox{for each $\epsilon \in (0,1]$ and each $(x,r) \in \operatorname{X} \times (0,R)$}.
\end{equation}
\end{Remark}

We will also use the following elementary assertion.
\begin{Lm}
\label{Lm.differentiation}
If the measure $\mu$ is uniformly locally doubling and $f \in L_{1}^{\rm loc}(\operatorname{X})$, then for each $\tau > 0$, it holds
\begin{equation}
\notag
\lim\limits_{r \to 0}\sup\limits_{x' \in B_{r}(x)}\fint\limits_{B_{\tau r}(x')}f(z)\,d\mu(z) = f(x) \quad \hbox{for $\mu$-a.e. $x \in \operatorname{X}$}.
\end{equation}
\end{Lm}

\begin{proof}
By Lemma \ref{Lm.volume_decay}, there is a constant depending on $\delta$ and $\operatorname{D}_{\mu}(1+\tau)$ only such that for all $x \in \operatorname{X}$ and $r \in (0,1/2)$ we have
\begin{equation}
\begin{split}
\notag
&\Bigl|\sup\limits_{x' \in B_{r}(x)}\fint\limits_{B_{\tau r}(x')}f(z)\,d\mu(z) - f(x)\Bigr|\\
&\le \sup\limits_{x' \in B_{r}(x)}\fint\limits_{B_{\tau r}(x')}|f(z)-f(x)|\,d\mu(z) \le C\fint\limits_{B_{(1+\tau)r}(x)}|f(z)-f(x)|\,d\mu(z).
\end{split}
\end{equation}
It remains to apply Lemma \ref{Lm.Lebesgue_point}.
\end{proof}

Given $q \in [1,\infty)$ and $R > 0$, for each $f \in L_{q}^{\rm loc}(\operatorname{X})$, we define the \textit{local maximal function} of $f$ at the scale $R$ by letting
\begin{equation}
\label{eqq.maximal_function}
\operatorname{M}^{\le R}_{q}[f](x):=\sup\limits_{r \in (0,R)}\Bigl(\fint\limits_{B_{r}(x)}|f(y)|^{q}\,d\mu(y)\Bigr)^{\frac{1}{q}}.
\end{equation}

The crucial fact concerning the local maximal functions is that the corresponding operator gives rise to a bounded endomorphism of $L_{p}(\operatorname{X})$
provided that $\mu$ has the doubling property. More precisely we have the following result (since the proof repeats almost verbatim arguments in the proof of Theorem 3.5.6 from \cite{HKST}, we omit details)
\begin{Prop}
\label{Prop.maximal_function}
If the measure $\mu$ is uniformly locally doubling measure and $p \in (1,\infty)$, then, for each $q \in [1,p)$ and each $R > 0$, there is a constant $C=C(q,R) > 0$ such that
\begin{equation}
\label{eqq.maximal_inequality}
\|\operatorname{M}^{\le R}_{q}[f]\|_{L_{p}(\operatorname{X})} \le C \|f\|_{L_{p}(\operatorname{X})}, \quad x \in \operatorname{X}.
\end{equation}
\end{Prop}

\subsection{Fundamental consequences of Poincar\'e-type inequalities}
Throughout this subsection we fix  $p \in [1,+\infty)$,
a metric measure space $\operatorname{X}=(\operatorname{X},\rho,\mu) \in \mathcal{A}_{p}$, and a complete metric space $\operatorname{Y}:=(\operatorname{Y},\operatorname{d})$.
We recall the fundamental $\operatorname{Lip}-\operatorname{lip}$-property of the $p$-admissible metric measure spaces.  This was firstly observed
by Keith in \cite{Keith} and then simplified in \cite{KlMac}. In this subsection we establish a similar result for metric-valued Lipschitz maps.

First of all we recall that a Borel function $\rho:\operatorname{X} \to [0,+\infty]$ is said to be an \textit{upper gradient} of a map $u:\operatorname{X} \to \operatorname{Y}$ if
for each rectifiable curve $\gamma:[a,b] \to \operatorname{X}$ it holds
\begin{equation}
\label{eqq.upper_gradient}
\operatorname{d}(u(a),u(b)) \le \int\limits_{0}^{l(\gamma)}\rho(\gamma_{s}(t))\,dt,
\end{equation}
where $l(\gamma)$ denotes the \textit{length} of $\gamma$ and $\gamma_{s}:[0,l(\gamma)] \to \gamma([a,b])$ is the so-called \textit{arc-length parametrization} of $\gamma$ (see Section 5.1 in \cite{HKST} for details).

\begin{Lm}
\label{Lm.upper_gradient}
If $u \in \operatorname{LIP}(\operatorname{X},\operatorname{Y})$, then $\operatorname{lip}[ u]$ is an upper gradient of $u$.
\end{Lm}

\begin{proof}
Clearly, for each $h \in \operatorname{LIP}_{1}(\operatorname{Y})$ we have $\operatorname{lip}[h \circ u](x) \le \operatorname{lip}[u](x)$ for all
$x \in \operatorname{X}$. According to Lemma 6.2.6 in \cite{HKST}, $\operatorname{lip}[h \circ u]$ is an upper gradient of $h\circ u$. Hence, for each rectifiable curve
$\gamma:[a,b] \to \operatorname{X}$, it holds
\begin{equation}
\notag
|(h\circ u)(\gamma(a))-(h\circ u)(\gamma(b))| \le \int\limits_{0}^{l(\gamma)}\operatorname{lip}[h\circ u](\gamma_{s}(t))\,dt.
\end{equation}
It remains to take the supremum over all $h \in \operatorname{LIP}_{1}(\operatorname{Y})$ in the left-hand side of the above inequality and apply Kuratowski imbedding theorem (see Section 4.1 in \cite{HKST}).
\end{proof}

The proof of the following important result is a direct consequence of Theorem 8.1.42 in \cite{HKST}, the Kuratowski imbedding result, and Lemma \ref{Lm.upper_gradient}.
We recall \eqref{eqq.Poincare}.
\begin{Prop}
\label{Prop.Poincare_maps}
Let $p \in [1,\infty)$ and $\operatorname{X} \in \mathcal{A}_{p}$. Then, for each $R > 0$,
there are constant $\operatorname{C} > 0$ and $\Lambda > 0$ (depending on $\operatorname{c}$ and $\lambda$) such that, for each $ u \in \operatorname{LIP}(\operatorname{X},\operatorname{Y})$,
\begin{equation}
\label{eqq.Poincare_maps}
\begin{split}
&\fint\limits_{B_{r}(x)}\fint\limits_{B_{r}(x)}\operatorname{d}(u(x'),u(x''))\,d\mu(x')\,d\mu(x'')\\
&\le \operatorname{C}r\Bigl(\fint\limits_{B_{\Lambda r}(x)}(\operatorname{lip}[ u](x'))^{p}\,d\mu(x')\Bigr)^{\frac{1}{p}} \quad \hbox{for all $(x,r) \in \operatorname{X} \times (0,R)$}.
\end{split}
\end{equation}
\end{Prop}

Now we are ready to establish the crucial analytical property. It's proof repeats almost verbatim the corresponding arguments from \cite{KlMac}.
We give a sketch for the reader convenience.
\begin{Prop}
\label{Prop.Lip_lip}
If $\operatorname{Y}=(\operatorname{Y},\operatorname{d})$ is separable, then there is a constant $\operatorname{K} > 0$ such that, for each $u \in \operatorname{LIP}(\operatorname{X},\operatorname{Y})$, it holds
\begin{equation}
\label{eqq.Lip_lip}
\operatorname{Lip}[u](x) \le \operatorname{K}\operatorname{lip}[u](x) \quad \hbox{for $\mu$-a.e. $x \in \operatorname{X}$}.
\end{equation}
\end{Prop}

\begin{proof}
Repeating with minor modifications arguments from Appendix A in \cite{KlMac} we deduce that the metric space $(\operatorname{X},\rho)$ is locally quasiconvex. More precisely,
for each $R > 0$, there is a constant $\sigma=\sigma(R) \geq 1$ such that for any $x_{1},x_{2}$ with $\rho(x_{1},x_{2}) \le R$ there is a rectifiable
curve $\gamma \in C([a,b],\operatorname{X})$ such that $\gamma(a)=x_{1}$, $\gamma(b) = x_{2}$ and $l(\gamma) \le \sigma \rho(x_{1},x_{2})$.

Since the map $u$ is Lipschitz the function $\operatorname{lip}[u]$ is Borel and bounded. Hence, by
Lemma \ref{Lm.differentiation} there is a set $\underline{\operatorname{X}} \subset \operatorname{X}$ with $\mu(\operatorname{X} \setminus \underline{\operatorname{X}}) = 0$
such that \eqref{eqq.differentiability_of_integral} holds for each $x \in \underline{\operatorname{X}}$ with $f$ replaced by $\operatorname{lip}[u]$.

Now we fix $\epsilon, \delta \in (0,\frac{1}{2})$ and $\underline{x} \in \underline{\operatorname{X}}$.
For any $r \in (0,1)$ and $y \in B_{r}(\underline{x})$, we fix an arbitrary finite sequence of points $\{x_{k}\}_{k=1}^{N}$, $N \in \mathbb{N}$ such that $\underline{x}=x_{1}$, $x_{N}=y$, $\rho(x_{i},x_{i+1}) \le \delta r$
and $N \le \frac{C}{\delta}$, where $C > 0$ depends on $\sigma$ only.
By Lemma \ref{Lm.differentiation}, there is $r(\underline{x},\epsilon) > 0$ such that, for each $r \in (0,r(\underline{x},\epsilon))$ and each $x \in B_{r}(\underline{x})$, we have
\begin{equation}
\notag
\begin{split}
&\fint\limits_{B_{\delta r}(x)}\fint\limits_{B_{\delta r}(x)}\operatorname{d}(u(x'),u(x''))\,d\mu(x')\,d\mu(x'')\\
&\le \operatorname{C}\delta r\Bigl(\fint\limits_{B_{\Lambda\delta r}(x)}(\operatorname{lip}[u](x'))^{p}\,d\mu(x')\Bigr)^{\frac{1}{p}} \le \operatorname{C}\delta r (\operatorname{lip}[u](\underline{x})+\epsilon).
\end{split}
\end{equation}
Hence, if $r \in (0,r(\underline{x},\epsilon))$, then by the triangle inequality we have
\begin{equation}
\notag
\begin{split}
&\operatorname{d}(u(\underline{x}),u(y)) \le \fint\limits_{B_{2\delta r}(x_{1})}\operatorname{d}(u(\underline{x}),u(x'))\,d\mu(x')+\sum\limits_{i=1}^{N}\fint\limits_{B_{2\delta r}(x_{i})}\operatorname{d}(u(x'),u(x''))\,d\mu(x')\,d\mu(x'')\\
&+ \fint\limits_{B_{2\delta r}(x_{N})}\operatorname{d}(u(y),u(x''))\,d\mu(x'') \le 4\delta r L_{u}+C N\delta r(\operatorname{lip}[u](x)+\epsilon)
\end{split}
\end{equation}
As a result, for each $r \in (0,r(\underline{x},\epsilon))$, we get
\begin{equation}
\notag
\operatorname{var}_{r}[u](\underline{x}) \le C(\operatorname{lip}[u](\underline{x})+\epsilon) + 4\delta L_{u},
\end{equation}
where $C > 0$ does not depend on $\epsilon$ and $\delta$. Hence, taking into account Lemma \ref{Lm.Lebesgue_point} and the arbitrariness in the choice of $\epsilon$ and $\delta$ we conclude.
\end{proof}

\subsection{Fine properties of Sobolev maps} Throughout this subsection
we fix $p \in (1,\infty)$, $\operatorname{X} \in \mathcal{A}_{p}$ and a complete pointed metric space $\operatorname{Y}=(\operatorname{Y},\operatorname{d},\underline{y})$.

We recall the so-called Lusin--Lipschitz property
of Sobolev spaces. It was proved in Section 8.2 of \cite{HKST} for the Banach-valued maps. The general case follows by the Kuratowski imbedding.
\begin{Prop}
\label{Prop.Lusin-Lipschitz}
Each $u \in W_{p}^{1}(\operatorname{X},\operatorname{Y})$ has the Lusin--Lipschitz property, i.e. there is a
sequence of Borel sets $\{V_{i}\}$ such that $\mu(\operatorname{X} \setminus \cup_{i}V_{i}) = 0$ and $ u|_{V_{i}}$ is Lipschitz for each
$i \in \mathbb{N}$.
\end{Prop}

We need a modification of the Sobolev imbedding theorem.
\begin{Prop}
\label{Prop.Sobolev_imb}
Given $R > 0$, let $q=q(R) \in (1,p)$ be the same as in Proposition \ref{Prop.Poincare_self_improvement}. Then, there is $C > 0$ such that, for each $f \in W_{p}^{1}(\operatorname{X})$ and each $r \in (0,R)$, it holds
\begin{equation}
\label{eqq.Sobolev_imb}
\operatorname{e}_{1,r}[f](x) \le C\operatorname{M}_{q}^{\le \widetilde{\lambda} r}[|Df|_{p}](x) \quad \hbox{for $\mu$-a.e. $x \in \operatorname{X}$}.
\end{equation}
\end{Prop}

\begin{proof}
By the main results of \cite{AGS}, given $f \in W_{p}^{1}(\operatorname{X})$ there is a sequence $\{f_{n}\} \subset \operatorname{LIP}(\operatorname{X})$
such that $\|f-f_{n}\|_{L_{p}(\operatorname{X})} \to 0$, $n \to \infty$, and, furthermore, $\|\operatorname{lip}f_{n}-|Df|_{p}\|_{L_{p}(\operatorname{X})} \to 0$, $n \to \infty$.
Since $L_{p}(\operatorname{X}) \subset L_{t}^{\rm loc}(\operatorname{X})$ for any $t \in [1,p]$, writing \eqref{eqq.Poincare} with $f$ replaced by $f_{n}$, $n \in \mathbb{N}$, and
passing to the limit as $n \to \infty$, we get for each $r \in (0,R)$,
\begin{equation}
\notag
\fint\limits_{B_{r}(x)}\fint\limits_{B_{r}(x)}|f(x')-f(x'')|\,d\mu(x')\,d\mu(x'') \le \widetilde{\operatorname{c}}r\Bigl(\fint\limits_{\widetilde{\lambda} r}(|Df|_{p}(z))^{q}\,d\mu(z)\Bigr)^{\frac{1}{q}}.
\end{equation}
Now, if $x$ is a Lebesgue point of $f$, then, using the standard telescopic-type arguments in combination with the uniformly locally doubling property of $\mu$ we get,
\begin{equation}
\notag
\operatorname{e}_{1,r}[f](x) \le\sum\limits_{i=0}^{\infty}\frac{C}{2^{i}}\Bigl(\fint\limits_{\frac{\widetilde{\lambda} r}{2^{i}}}(|Df|_{p})^{q}(z)\,d\mu(z)\Bigr)^{\frac{1}{q}} \le C\operatorname{M}_{q}^{\le \widetilde{\lambda} r}[|Df|_{p}](x).
\end{equation}
The proof is complete.
\end{proof}

Now we can prove a rather standard locality property of the upper energy.
\begin{Prop}
\label{Prop.locality}
Let $u_{1} \in W_{p}^{1}(\operatorname{X},\operatorname{Y})$, $u_{2} \in \operatorname{LIP}(\operatorname{X},\operatorname{Y})$ and $U \subset \operatorname{X}$  be such that $\mu(U) > 0$ and $ u_{1}|_{U} = u_{2}|_{U}$.
Then, for each $q \in [1,p]$, for $\mu$-a.e. $x \in U$,
\begin{equation}
\label{eqq.locality_energy''}
\varlimsup\limits_{r \to 0}|\operatorname{e}_{q,r}[u_{1}](\underline{x})-\operatorname{e}_{q,r}[u_{2}](\underline{x})|=0.
\end{equation}
In particular,
\begin{equation}
\label{eqq.locality_energy}
\underline{\operatorname{e}}_{q}[u_{1}](x) = \underline{\operatorname{e}}_{q}[u_{2}](x), \quad \overline{\operatorname{e}}_{q}[u_{1}](x) = \overline{\operatorname{e}}_{q}[u_{2}](x) \quad \hbox{for $\mu$-a.e. $x \in U$}.
\end{equation}
\end{Prop}

\begin{proof}
According to Proposition \ref{Prop.different_approaches} we have $u \in \operatorname{HS}^{1}_{p}(\operatorname{X},\operatorname{Y})$.
Let $\underline{U} \subset U$ be the set of all density points of $U$ where $G_{1}[u_{1}]$ is finite, \eqref{eqq.differentiability_of_integral'} holds with $E$ replaced by $U$ and with $f$ replaced by both
$G_{1}[u_{1}]$ and $G_{1}[u_{2}]=L_{u_{2}}$. By Remark \ref{Rem.differentiation_integarl} we have $\mu(U \setminus \underline{U})=0$.
As a result, given $\underline{x} \in \underline{U}$, using the triangle inequality we get
\begin{equation}
\label{eqq.forgotten_3.14}
\begin{split}
&\varlimsup\limits_{r \to 0}|\operatorname{e}_{q,r}[u_{1}](\underline{x})-\operatorname{e}_{q,r}[u_{2}](\underline{x})| \le \varlimsup\limits_{r \to 0}\sum\limits_{i=1}^{2}
\Bigl(\frac{1}{\mu(B_{r}(\underline{x}))}\int\limits_{B_{r}(\underline{x}) \setminus U}\frac{\operatorname{d}^{q}(u_{i}(\underline{x}), u_{i}(x'))}{r^{q}}\,d\mu(x')\Bigr)^{\frac{1}{q}}\\
&\le \varlimsup\limits_{r \to 0}\sum\limits_{i=1}^{2}\Bigl[\Bigl(\int\limits_{B_{r}(\underline{x}) \setminus U}\frac{1}{\mu(B_{r}(\underline{x}))}(G_{1}[ u_{i}](x'))^{p}\,d\mu(x')\Bigr)^{\frac{1}{p}}+\Bigl(\frac{\mu(B_{r}(\underline{x})\setminus U)}{\mu(B_{r}(\underline{x}))}\Bigr)^{\frac{1}{q}}G_{1}[u_{i}](\underline{x})\Bigr]=0.
\end{split}
\end{equation}
Now the second equality in \eqref{eqq.locality_energy} is an easy consequence of \eqref{eqq.forgotten_3.14} and the obvious inequality
\begin{equation}
\notag
|\overline{\operatorname{e}}_{q}[u_{1}](\underline{x})-\overline{\operatorname{e}}_{q}[u_{2}](\underline{x})| \le \varlimsup\limits_{r \to 0}|\operatorname{e}_{q,r}[u_{1}](\underline{x})-\operatorname{e}_{q,r}[u_{2}](\underline{x})|.
\end{equation}
To prove the first equality, we fix a sequence $r_{n} \downarrow 0$, $n \to \infty$ such that
\begin{equation}
\label{eqq.forgotten_3.15}
\lim\limits_{n \to \infty}\operatorname{e}_{q,r_{n}}[u_{1}](\underline{x}) = \underline{\operatorname{e}}_{q}[u_{1}](\underline{x}).
\end{equation}
Hence, by \eqref{eqq.forgotten_3.14} and \eqref{eqq.forgotten_3.15} we have
\begin{equation}
\notag
\underline{\operatorname{e}}_{q}[u_{2}](\underline{x}) \le \varliminf\limits_{n \to \infty}\operatorname{e}_{q,r_{n}}[u_{2}](\underline{x}) \le \underline{\operatorname{e}}_{q}[u_{1}](\underline{x})+
\varliminf\limits_{n \to \infty}|\operatorname{e}_{q,r_{n}}[u_{2}](\underline{x})-\operatorname{e}_{q,r_{n}}[u_{1}](\underline{x})| \le \underline{\operatorname{e}}_{q}[u_{1}](\underline{x}).
\end{equation}
Due to the symmetry of our arguments, the opposite inequality follows similarly.

The proof is complete.
\end{proof}

The observation made above leads to a useful lemma.
\begin{Lm}
\label{Lm.auxiliary_norm}
There is a constant $C > 0$ such that for each $u \in W^{1}_{p}(\operatorname{X},\operatorname{Y})$, it holds
\begin{equation}
\frac{1}{C}|Du|_{p}(x) \le \overline{\operatorname{e}}_{p}[u](x) \le C|Du|_{p}(x) \quad \hbox{for $\mu$-a.e. $x \in \operatorname{X}$}.
\end{equation}
\end{Lm}

\begin{proof}
By Proposition \ref{Prop.Poincare_self_improvement}
there are $\underline{q} \in (1,p)$, $\underline{\lambda} \geq 1$ and $\underline{c} > 0$ such that \eqref{eqq.Poincare} holds true with $R = 1$
and $q$, $\lambda$, $\operatorname{c}$ replaced by $\underline{q}$, $\underline{\lambda}$, $\underline{\operatorname{c}}$, respectively. Hence, repeating
almost verbatim arguments from the proof of Proposition 2.17 in \cite{GT} we derive existence of a constant $\underline{C} > 0$ depending
on $\operatorname{D}_{\mu}(\underline{\lambda})$ and $\underline{\lambda}$, $\underline{\operatorname{c}}$ such that $G:=\underline{C}\operatorname{M}^{\le 2 \underline{\lambda}}_{\underline{q}}[|Du|_{p}]$ is a Hajlasz gradient of $u$ at the scale $1$.
By Proposition \ref{Prop.different_approaches} and Remark \ref{Rem.differentiation_integarl} we have the following inequality
\begin{equation}
\overline{\operatorname{e}}_{p}[u](x) \le C|Du|_{p}(x) \quad \hbox{for $\mu$-a.e. $x \in \operatorname{X}$},
\end{equation}
where the constant $C > 0$ does not depend on $u$ and $x$.

To prove the opposite estimate we fix a sequence $\{r_{j}\}$ of positive numbers such that
$r_{j} \downarrow 0$, $j \to \infty$ and the sequence $\{\operatorname{e}_{p,r_{j}}[u]\}$ weakly converges to some $g$ in $L_{p}(\operatorname{X})$ (the fact that at least one
such $g$ exists follows from the definition of $\operatorname{KS}^{1}_{p}(\operatorname{X},\operatorname{Y})$, Proposition \ref{Prop.different_approaches}, and the reflexivity of $L_{p}(\operatorname{X})$). We claim that
\begin{equation}
g(x) \le \overline{\operatorname{e}}_{p}[u](x) \quad \hbox{for $\mu$-a.e. $x \in \operatorname{X}$}.
\end{equation}
Indeed, this follows from rather routine arguments. We set $\overline{\operatorname{e}}_{p,r}[ u](x):=\sup_{r' \in (0,r]}\operatorname{e}_{p,r'}[u](x)$ for each
$r > 0$ and each $x \in \operatorname{X}$. Since $\overline{\operatorname{e}}_{p}[u]$ is measurable and $\overline{\operatorname{e}}_{p}[u](x) < +\infty$ for $\mu$-a.e. $x \in \operatorname{X}$,
by the monotonicity of the sequence $\{\overline{\operatorname{e}}_{p,r_{j}}[ u]\}$ and Egorov's theorem, for each set $E \subset \operatorname{X}$
with $\mu(E) \in (0,+\infty)$, for each $\epsilon > 0$ there is a compact set $K \subset E$ such that $\mu(E \setminus K) < \epsilon$ and $\{\overline{\operatorname{e}}_{p,r_{j}}[u]\}$ converges uniformly
to $\overline{\operatorname{e}}_{p}[u]$. Since $\operatorname{e}_{p,r_{j}}[u](x) \le \overline{\operatorname{e}}_{p,r_{j}}[u](x)$ for all $x \in \operatorname{X}$ and all $j \in \mathbb{N}$, by the definition of the weak convergence in $L_{p}(\operatorname{X})$ we get
$$
\int\limits_{K}g(x)\,d\mu(x) \le \int\limits_{K}\overline{\operatorname{e}}_{p}[u](x)\,d\mu(x).
$$
Since $E$ and $K$ were chosen arbitrarily, the claim follows.
\end{proof}

\section{Nice properties of energy densities}
Informally speaking, the aim of the present section is to establish an analog of Proposition \ref{Prop.Lip_lip} for energy densities of Lipschitz maps.
The ideological core is quite similar to that of given in Section 6 of \cite{KlMac}. Of course, the results of this section are not surprising for the experts. However, it is difficult to find
the precise analogs in the literature.

\begin{Th}
\label{Th.energy_density}
Let $p \in [1,\infty)$ and $\operatorname{X} \in \mathcal{A}_{p}$. Let $\operatorname{Y}=(\operatorname{Y},\operatorname{d},\underline{y})$ be a pointed complete metric space.
Then there is a constant $\operatorname{K}_{p} \geq 1$
such that if $u \in \operatorname{LIP}(\operatorname{X},\operatorname{Y})$, then
\begin{equation}
\label{eqq.energy_densities_comparison}
\overline{\operatorname{e}}_{p}[u](x) \le \operatorname{Lip}[u](x) \le \operatorname{K}_{p} \underline{\operatorname{e}}_{1}[u](x) \le \operatorname{K}_{p} \underline{\operatorname{e}}_{p}[u](x) \quad \hbox{for $\mu$-a.e. $x \in \operatorname{X}$}.
\end{equation}
\end{Th}

\begin{proof}
The first inequality in \eqref{eqq.energy_densities_comparison}
is an immediate consequence of the corresponding definitions. The third inequality in \eqref{eqq.energy_densities_comparison} follows from the H\"older inequality.
It is sufficient to establish the second inequality.
We fix $u \in \operatorname{LIP}(\operatorname{X},\operatorname{Y})$ and
split the rest of the proof into several elementary steps.

\textit{Step 1.}
Since $u$ is Lipschitz, it is clear that $\operatorname{lip}[u]$ is locally integrable.
Hence, by Lemma \ref{Lm.differentiation} there is a set $\underline{\operatorname{X}}$ with $\mu(\operatorname{X} \setminus \underline{\operatorname{X}})=0$
such that, for each $\underline{x} \in \underline{\operatorname{X}}$, for each $N \in \mathbb{N}$, we have
\begin{equation}
\label{eqq.4.forgotten}
\lim\limits_{r \to +0}\sup\limits_{x' \in B_{r}(\underline{x})}\fint\limits_{B_{\frac{\Lambda \delta r}{2^{i}}}(x')}(\operatorname{lip}u(x))^{p}\,dx = (\operatorname{lip}u(\underline{x}))^{p} \quad \hbox{for all $1 \le i \le N$}.
\end{equation}
If $\operatorname{Lip}[u](\underline{x})=0$, then the second inequality in \eqref{eqq.energy_densities_comparison} is trivial. Hence,
we may assume that $\operatorname{Lip}[u](\underline{x}) > 0$. By Proposition \ref{Prop.Lip_lip} this implies that $\operatorname{lip}[u](\underline{x}) > 0$.

\textit{Step 2.}
We fix an arbitrary $\underline{x} \in \operatorname{X}$.
Taking into account Remark \ref{Rem.different_variations}, given $r \in (0,1)$, we fix $y_{r} \in B_{r}(\underline{x})$ such that
\begin{equation}
\notag
\operatorname{var}_{r}[u](\underline{x}) \le \frac{3\operatorname{d}(u(\underline{x}), u(y_{r}))}{r}.
\end{equation}
For each $\delta \in (0,1)$, we clearly have
\begin{equation}
\label{eqq.4.4'}
\begin{split}
&\operatorname{d}(u(\underline{x}),u(y_{r}))\le \fint\limits_{B_{\delta r}(y_{r})}\operatorname{d}(u(\underline{x}),u(z))\,d\mu(z) + \fint\limits_{B_{\delta r}(y_{r})}\operatorname{d}(u(y_{r}),u(z))\,d\mu(z).
\end{split}
\end{equation}

\textit{Step 3.}
By the standard telescopic-type arguments in combination with Proposition \ref{Prop.Poincare_maps}, for each $n \in \mathbb{N}$, we have
\begin{equation}
\label{eqq.4.5'}
\begin{split}
&\fint\limits_{B_{\delta r}(y_{r})}\frac{\operatorname{d}(u(y_{r}),u(z))}{r}\,d\mu(z)\\
&\le C\sum\limits_{i=0}^{n}\frac{\delta}{2^{i}}\Bigl(\fint\limits_{B_{\frac{\Lambda\delta r}{2^{i}}}(y_{r})}(\operatorname{lip}[u](z))^{p}\,d\mu(z)\Bigr)^{\frac{1}{p}} + \fint\limits_{B_{\frac{\delta r}{2^{n}}}(y_{r})}\frac{\operatorname{d}(u(y_{r}),u(z))}{r}\,d\mu(z),
\end{split}
\end{equation}
where the constant $C > 0$ does not depend on $ u$, $n$, $\underline{x}$ and $y_{r}$.

\textit{Step 4.}
Given $\epsilon > 0$ we take $n \in \mathbb{N}$ so large that $\frac{\delta}{2^{n}} < \epsilon$. As a result, taking into account \eqref{eqq.4.forgotten} and the global Lipschitz continuity of $u$, we can
take $\overline{r}(\underline{x}) > 0$ so small that
\begin{equation}
\label{eqq.4.6'}
\fint\limits_{B_{\delta r}(y_{r})}\frac{\operatorname{d}(u(y_{r}),u(z))}{r}\,d\mu(z) \le C\delta\operatorname{lip}[u](\underline{x}) + \epsilon L_{u} \quad \hbox{for all $r \in (0,\overline{r}(\underline{x}))$},
\end{equation}
where the constant $C > 0$ is the same as in \eqref{eqq.4.5'}.

\textit{Step 5.}
We take $\delta > 0$ and then $\epsilon > 0$ so small that (we recall that $\operatorname{lip}[u](\underline{x}) > 0$)
\begin{equation}
3Cr\delta\operatorname{lip}[u](x) + 3\epsilon r L_{u} < \frac{r}{2}\operatorname{lip}[u](\underline{x}).
\end{equation}
As a result, combining \eqref{eqq.4.4'}--\eqref{eqq.4.6'} we get for each $r \in (0,\overline{r}(\underline{x}))$
\begin{equation}
\begin{split}
\label{eqq.4.7'}
&\operatorname{var}_{r}[u](\underline{x}) \le 3\fint\limits_{B_{\delta r}(y_{r})}\frac{\operatorname{d}(u(\underline{x}),u(z))}{r}\,d\mu(z) + \frac{1}{2}\operatorname{lip}[ u](\underline{x})\\
&\le \frac{3\mu(B_{r}(x))}{\mu(B_{\delta r}(y_{r}))}\operatorname{e}_{1,3r}[u](\underline{x})+\frac{1}{2}\operatorname{lip}[u](\underline{x}).
\end{split}
\end{equation}
Taking the limit infimum in \eqref{eqq.4.7'}  in combination with \eqref{eqq.doubling_measure}
we deduce
\begin{equation}
\notag
\operatorname{lip}[u](\underline{x}) \le 6(\operatorname{D}_{\mu}(1))^{\frac{2}{\delta}+1}\varliminf\limits_{r \to +0}\operatorname{e}_{1,3r}[u](\underline{x}).
\end{equation}
Since $\delta > 0$ depends on $p$ but doest not depend on $u$, $\epsilon > 0$ and $\underline{x}$, the required inequality follows from Proposition \ref{Prop.Lip_lip} with $\operatorname{K}_{p}:=6(\operatorname{D}_{\mu}(1))^{\frac{2}{\delta}+1}\operatorname{K}$.

The proof is complete.
\end{proof}

The following observation can be interesting in itself. It should be compared with Theorem 3.13 in \cite{GT}, which provides much more
precise local behavior of $\operatorname{e}_{p,r}[u]$ in the context of strongly rectifiable metric measure spaces.
\begin{Ca}
\label{Ca.Sobolev_energy_densities_comparison}
Let $p \in [1,\infty)$ and $\operatorname{X} \in \mathcal{A}_{p}$. Let $\operatorname{Y}=(\operatorname{Y},\operatorname{d},\underline{y})$ be a pointed complete metric space.
Then, for each $u \in W_{p}^{1}(\operatorname{X},\operatorname{Y})$,
\begin{equation}
\label{eqq.Sobolev_energy_densities_comparison}
\overline{\operatorname{e}}_{p}[u](x) \le \operatorname{K}_{p} \underline{\operatorname{e}}_{1}[u](x) \quad \hbox{for $\mu$-a.e. $x \in \operatorname{X}$}.
\end{equation}
\end{Ca}

\begin{proof}
We recall Lemma \ref{Lm.Frechet}. Replacing $u$ by $\iota \circ u$ without loss of generality we may assume that $\operatorname{Y}$ is a complete separable subset of $l_{\infty}$ and $\underline{y}=0$.
Let the sequence $\{V_{i}\}$ of subsets of $\operatorname{X}$ be the same as in Proposition \ref{Prop.Lusin-Lipschitz}. Given $i \in \mathbb{N}$, let $\overline{u}_{i} \in \operatorname{LIP}(\operatorname{X},l_{\infty})$
be the extension of $u$ from $V_{i}$ to $\operatorname{X}$ that exists by Lemma \ref{Lm.McShane_Whitney}.
Combining Proposition \ref{Prop.locality} with Theorem \ref{Th.energy_density} we get existence of a set $\underline{V}_{i} \subset V_{i}$ with $\mu(V_{i} \setminus V_{i}) = 0$ such that
\begin{equation}
\notag
\overline{\operatorname{e}}_{p}[u](x) = \overline{\operatorname{e}}_{p}[\overline{u}_{i}](x) \le \operatorname{K}_{p} \underline{\operatorname{e}}_{1}[\overline{u}_{i}](x) = \operatorname{K}_{p} \underline{\operatorname{e}}_{1}[u](x) \quad \hbox{for all $x \in \underline{V}_{i}$}.
\end{equation}
Since $\mu(\operatorname{X} \setminus \cup_{i \in \mathbb{N}}\underline{V}_{i})=0$, the claim follows.
\end{proof}

\section{Fine covering observations}
We fix a metric measure space $\operatorname{X}=(\operatorname{X},\rho,\mu)$ such that the measure
$\mu$ is uniformly locally doubling. Furthermore, we fix
a metric space $\operatorname{Y}=(\operatorname{Y},\operatorname{d})$.

\begin{Def}
\label{Def.asymptotically_doubling}
We say that a Borel measure $\mathfrak{m}$ on a metric space $(\operatorname{Z},\operatorname{d}_{\operatorname{Z}})$ is uniformly weakly asymptotically doubling if $\mathfrak{m}(B_{r}(z)) > 0$ for all
$(z,r) \in \operatorname{Z} \times (0,+\infty)$ and, for each $\tau > 1$,
\begin{equation}
\label{eqq.asymptotically_doubling}
\overline{\operatorname{AD}}_{\mathfrak{m}}(\tau):=\sup\limits_{z \in \operatorname{Z}}\varliminf\limits_{r \to 0}\frac{\mathfrak{m}(B_{\tau r}(z))}{\mathfrak{m}(B_{r}(z))} < +\infty.
\end{equation}
\end{Def}

\begin{Remark}
\label{Remark.degenerate_slowly}
By elementary considerations it is clear that $\operatorname{AD}_{\mathfrak{m}}(\tau) < +\infty$ if and only if
\begin{equation}
\notag
\underline{\operatorname{AD}}_{\mathfrak{m}}(\tau):=\inf\limits_{z \in \operatorname{Z}}\varlimsup\limits_{r \to 0}\frac{\mathfrak{m}(B_{r}(z))}{\mathfrak{m}(B_{\tau r}(z))} > 0.
\end{equation}
\end{Remark}

\begin{Prop}
\label{Prop.Reduction_to_Vitali}
Let a nonnegative $f \in L_{1}(\operatorname{X})$,
a Borel representative $\overline{f}$ of $f$, and a Borel subset $U$ of $\operatorname{X}$
are such that:

\begin{itemize}
\item[\(($1$)\)] $\mu(U) > 0$;

\item[\(($2$)\)] $U \subset \{x \in \operatorname{X}:\overline{f}(x) > 0\}$.

\end{itemize}
If $u \in \operatorname{LIP}(U,\operatorname{Y})$, then, the measure $\mathfrak{m}=u_{\sharp}(f\chi_{U}\mu)$
is uniformly weakly asymptotically doubling on the metric space $\operatorname{Z}:=(u(\underline{U}),\operatorname{d})$, where
\begin{equation}
\label{eqq.5.4''}
\underline{U}:=\Bigl\{x \in U:\lim\limits_{r \to +0}\fint_{B_{r}(x)}(\chi_{U}f)(x')\,d\mu(x')=\overline{f}(x)\Bigr\}.
\end{equation}
\end{Prop}

\begin{proof}
By Remark \ref{Rem.differentiation_integarl}
we have $\mu(U \setminus \underline{U})=0$.
Hence, $\mathfrak{m}=u_{\sharp}(f\chi_{\underline{U}}\mu)$. We split the rest
of the proof into several steps.

\textit{Step 1.} Since $f \in L_{1}(\operatorname{X})$, the measure $\mathfrak{m}$ is finite, i.e.
\begin{equation}
\label{eqq.3.2}
\mathfrak{m}(\operatorname{Z}) < +\infty.
\end{equation}

\textit{Step 2.} Obviously, $ u \in \operatorname{LIP}_{L}(\operatorname{X},\operatorname{Y})$ for some $L > 0$.
Then, for each $z \in \operatorname{Z}$, each $x \in  u^{-1}(z)$, and each $r > 0$, we have (note that
here and below in the proof the symbol $B_{r}(z)$ denotes an open ball in $\operatorname{Z}$ centered at $z$ with radius $r$)
\begin{equation}
\notag
u^{-1}(B_{r}(z)) \supset B_{\frac{r}{L}}(x) \cap \underline{U}.
\end{equation}
It is clear from the condition (2) of the Proposition and \eqref{eqq.5.4''} that
\begin{equation}
\overline{f}(x) > 0 \quad \hbox{for each $z \in \operatorname{Z}$, for every $x \in  u^{-1}(z) \cap \underline{U}$.}
\end{equation}

\textit{Step 3.}
Given $z \in u(\underline{U})$, we fix an arbitrary $x_{z} \in  u^{-1}(z) \cap \underline{U}$.
Since $\mu$ is uniformly locally doubling, by Lemma \ref{Lm.volume_decay} and \eqref{eqq.5.4''}, given $\theta > \underline{\theta}_{\mu}$,  there is $\delta=\delta(z,x_{z}) > 0$ such that
\begin{equation}
\begin{split}
\label{eqq.3.3}
&\mathfrak{m}(B_{r}(z)) \geq \int\limits_{B_{\frac{r}{L}}(x_{z}) \cap U}f(x)\,d\mu(x) \geq C\Bigl(\frac{r}{L}\Bigr)^{\theta}\overline{f}(x_{z}) \quad \hbox{for all $r \in (0,\delta)$}.
\end{split}
\end{equation}

\textit{Step 4.}
If we assume on the contrary that the measure $u_{\sharp}(f\chi_{U}\mu)$ fails to satisfy \eqref{eqq.asymptotically_doubling}, then by Remark \ref{Remark.degenerate_slowly} there exist
a parameter $\underline{\tau} > 1$  such that for each $\theta' > \theta$ there are $z \in \operatorname{Z}$ and small enough $\delta'=\delta'(\theta',z) > 0$ for which
\begin{equation}
\label{eqq.3.4}
\frac{\mathfrak{m}\Bigl(B_{\frac{r}{\underline{\tau}}}(z)\Bigr)}{\mathfrak{m}(B_{r}(z))} \le \Bigl(\frac{1}{\underline{\tau}}\Bigr)^{\theta'} \quad \hbox{for all $r \le \delta'$}.
\end{equation}

\textit{Step 4.}
Taking into account \eqref{eqq.3.2}, given $z \in \operatorname{Z}$, we deduce that \eqref{eqq.3.3} and \eqref{eqq.3.4} lead to the existence of the
constants $C_{1}(z),C_{2}(z) > 0$ such that (we put $\overline{\delta}:=\min\{\delta,\delta'\}$)
\begin{equation}
\label{eqq.5.6''}
C_{1}(z)r^{\theta} \le \mathfrak{m}(B_{r}(z)) \le C_{2}(z)r^{\theta'} \quad \hbox{for all $r \in (0,\overline{\delta})$}.
\end{equation}
Clearly, both inequalities in \eqref{eqq.5.6''} cannot hold simultaneously for small enough $r > 0$.

We arrive at a contradiction and complete the proof.
\end{proof}

\section{Tracking energies of Sobolev maps by Lipschitz functions}

Throughout the section the following data are assumed to be fixed:

\begin{itemize}
\item[\(($D.6.1$)\)] a parameter $p \in (1,\infty)$;

\item[\(($D.6.2$)\)] a metric measure space $\operatorname{X}=(\operatorname{X},\rho,\mu) \in \mathcal{A}_{p}$ which is geodesic at the scale 1;

\item[\(($D.6.3$)\)] a pointed complete metric space $\operatorname{Y}=(\operatorname{Y},\operatorname{d},\underline{y})$.
\end{itemize}

We start with the following elementary and standard fact.

\begin{Lm}
\label{Lm.average}
Let $\operatorname{Z}=(\operatorname{Z},\operatorname{d}_{\operatorname{Z}},\mathfrak{m})$ be a metric measure space.
Then, given a nonnegative function $g \in L_{1}^{\rm loc}(\operatorname{Z})$, for each nonempty ball $B \subset \operatorname{Z}$ with $\mathfrak{m}(B) > 0$ we have
\begin{equation}
\label{eqq.6.1}
\int\limits_{B}g(y)\,d\mathfrak{m}(y) \le 2 \int\limits_{\{g > \frac{g_{B}}{2}\} \cap B}g(y)\,d\mathfrak{m}(y),
\end{equation}
where $g_{B}:=\fint\limits_{B}g(z)\,d\mathfrak{m}(z)$.
\end{Lm}

\begin{proof}
Assume on the contrary that the opposite inequality in $\eqref{eqq.6.1}$ holds for some ball $B \subset \operatorname{Z}$ with $\mathfrak{m}(B) > 0$. Hence, we immediately arrive at a contradiction:
\begin{equation}
\notag
\begin{split}
&g_{B}=\frac{1}{\mathfrak{m}(B)}\int\limits_{g > \frac{g_{B}}{2}}g(z)\,d\mathfrak{m}(z) +\frac{1}{\mathfrak{m}(B)}\int\limits_{g \le \frac{g_{B}}{2}}g(z)\,d\mathfrak{m}(z) < \frac{g_{B}}{2}+\frac{g_{B}}{2}=g_{B}.
\end{split}
\end{equation}
The proof is complete.
\end{proof}

Given $u \in L_{1}^{\rm loc}(\operatorname{X},\operatorname{Y})$, for each $k \in \mathbb{Z}$ and each $r > 0$, we put
\begin{equation}
\label{eqq.image_distance}
\operatorname{S}_{k,r}[u]:=\{x \in \operatorname{X}: r\operatorname{e}_{1,r}[ u](x) \in (2^{-k-1},2^{-k}]\}.
\end{equation}
Furthermore, $k \in \mathbb{Z}$, it will be convenient to work with the set
\begin{equation}
\label{eqq.image_distance'}
\operatorname{S}_{k}[u]:=\bigcup\limits_{r \in (0,1)}\operatorname{S}_{k,r}[u].
\end{equation}

We should make an important technical observation.
\begin{Lm}
\label{Lm.exhaustive}
For each $u \in L_{1}^{\rm loc}(\operatorname{X},\operatorname{Y})$, the following properties hold:

\begin{itemize}

\item[\((1)\)] $\mu(\operatorname{S}_{k}[u] \setminus \operatorname{S}_{k+j}[u]) = 0$ for each $k \in \mathbb{Z}$ and each $j \in \mathbb{N}$;

\item[\((2)\)] the map $u$ has a constant representative if and only if $\mu(\operatorname{S}_{k}[u]) = 0$ for all $k \in \mathbb{Z}$;

\item[\((3)\)] it holds
\begin{equation}
\notag
\mu\Bigl(\{x\in \operatorname{X}:\overline{\operatorname{e}}_{1}[u](x) > 0\} \setminus \bigcup_{k \in \mathbb{Z}}\operatorname{S}_{k}[u]\Bigr)=0.
\end{equation}

\end{itemize}

\end{Lm}

\begin{proof}
Note that (1) is obvious if $\operatorname{S}_{k}[u] = \emptyset$. Otherwise, by Lemma \ref{Lm.Lebesgue_point} we have
\begin{equation}
\notag
\lim\limits_{r \to +0}r\operatorname{e}_{1,r}[u](x) = 0, \quad \hbox{for $\mu$-a.e. $x \in \operatorname{S}_{k}[u]$}.
\end{equation}
It remains to note that the map $r \mapsto r\operatorname{e}_{1,r}[u](x)$ is continuous on $(0,1)$ (this is an elementary consequence
of Remark \ref{Rem.volume_decay}).
Property (2) is obvious, whereas to prove (3) it is sufficient to note that for each $x \in \operatorname{X}$ with $\overline{\operatorname{e}}_{1}[u](x) > 0$, there are $k \in \mathbb{Z}$ and $r \in (0,1)$ such that $r\operatorname{e}_{1,r}[u](x) \in (2^{-k-1},2^{-k}]$. Indeed, otherwise, since $u \in L_{1}^{\rm loc}(\operatorname{X},\operatorname{Y})$, we would have $r\operatorname{e}_{1,r}[u](x) = 0$ for all $r \in (0,1)$, and consequently $\overline{\operatorname{e}}_{1}[u](x) = 0$.
\end{proof}

If $x \in \operatorname{X}$, then we put
\begin{equation}
\label{eqq.minimal_radius}
r_{k}[u](x):=
\begin{cases}
\inf\{r > 0:r\operatorname{e}_{1,r}[u](x) \in (2^{-k-1},2^{-k}]\}, \quad x \in \operatorname{S}_{k}[u];\\
+\infty, \quad x \notin \operatorname{S}_{k}[u].
\end{cases}
\end{equation}

\begin{Remark}
\label{Rem.informal_explanation}
Given $r > 0$, the set $\operatorname{S}_{k,r}[u]$ can be informally considered as the set of all centers of balls with radii $r$ on which
the average oscillations of $ u$ are approximately equal to $2^{-k}$. At the same time the parameters $r_{k}[u](x)$, $x \in \operatorname{X}$, allow
to treat the minimal radii of balls on each of which the average oscillations of $ u$ are comparable with $2^{-k}$.
\end{Remark}

Given $u \in W_{p}^{1}(\operatorname{X},\operatorname{Y})$, we put
\begin{equation}
\label{eqq.Lebesgue_sets_of_differential}
\mathcal{D}_{p}[u]:=\Bigl\{x \in \operatorname{X}:\varlimsup\limits_{r \to +0}\fint\limits_{B_{r}(x)}|Du|^{p}_{p}(x) > 0\Bigr\}.
\end{equation}

\begin{Prop}
\label{Prop.minimal_radius_tends_to_zero}
Given $u \in W_{p}^{1}(\operatorname{X},\operatorname{Y})$, for $\mu$-a.e. $x \in \mathcal{D}_{p}[u]$, there is $k_{x} \in \mathbb{N}$ such that 
the sequence $\{r_{k}[u](x)\}_{k=k_{x}}^{\infty}$ is nonincreasing and
\begin{equation}
\notag
r_{k}[u](x) \downarrow 0, \quad k \to +\infty.
\end{equation}
\end{Prop}

\begin{proof}
Given $q \in [1,p]$, we put (we take into account that $u \in L_{p}(\operatorname{X},\operatorname{Y})$, and hence $u \in L_{q}^{\rm loc}(\operatorname{X},\operatorname{Y})$)
\begin{equation}
\label{eqq.E_set}
\mathcal{E}_{q}[u]:=\bigcup\limits_{k \in \mathbb{Z}}\bigcap\limits_{j \geq k}\operatorname{S}_{j}[u] \cap \{x \in \operatorname{X}: \overline{\operatorname{e}}_{q}[u](x) > 0 \hbox{ and }\lim\limits_{r \to +0}r\operatorname{e}_{q,r}[u](x)=0\}.
\end{equation}
Note that by Lemma \ref{Lm.Lebesgue_point}, Lemma
\ref{Lm.auxiliary_norm}, and Corollary \ref{Ca.Sobolev_energy_densities_comparison} we clearly have
\begin{equation}
\label{eqq.different_level_sets}
\mu(\mathcal{E}_{1}[u] \triangle \mathcal{D}_{p}[u]) = 0.
\end{equation}
Fix an arbitrary $\underline{x} \in \mathcal{E}_{1}[u]$. Clearly, there is $k_{\underline{x}} \in \mathbb{N}$ such that 
$r_{k}[u](x) < +\infty$ for all $k \geq k_{\underline{x}}$. Assuming that $\varlimsup_{k \to \infty}r_{k}[u](\underline{x}) > 0$ we get existence of a number $\underline{r} > 0$
together with a subsequence $\{r_{k_{j}}[u](\underline{x})\}$ such that  $r_{k_{j}}[u](\underline{x}) \in (\underline{r},2\underline{r})$ for all $j \in \mathbb{N}$. By \eqref{eqq.minimal_radius}
this implies existence of a sequence $\{r'_{j}\} \subset (\underline{r},2\underline{r})$ such that
\begin{equation}
\notag
\operatorname{e}_{1,\underline{r}}[u](\underline{x}) \le 2\frac{\mu(B_{2\underline{r}}(\underline{x}))}{\mu(B_{\underline{r}}(\underline{x}))}\operatorname{e}_{1,r'_{j}}[u](\underline{x}) \le \operatorname{D}_{\mu}(\underline{r})\frac{2^{-k_{j}+2}}{\underline{r}} \quad \hbox{for all $j \in \mathbb{N}$}.
\end{equation}
Hence, $\operatorname{e}_{1,\underline{r}}[u](\underline{x}) = 0$. Thus, $u(x)=u(\underline{x})$ for $\mu$-a.e. $x \in B_{\underline{r}}(\underline{x})$ and, consequently,
$\overline{\operatorname{e}}_{1}[u](\underline{x}) = 0$. This leads to the contradiction with \eqref{eqq.E_set}.
As a result, we get
\begin{equation}
\label{eqq.limit_forgotten}
\lim\limits_{k \to \infty}r_{k}[u](\underline{x}) =0 \quad \hbox{for each $\underline{x} \in \mathcal{E}_{1}[u]$}.
\end{equation}
Furthermore, given $\underline{x} \in \mathcal{E}_{1}[u]$, combining \eqref{eqq.limit_forgotten} with the continuity of the map $r \mapsto r\operatorname{e}_{1,r}[u](\underline{x})$ on $(0,1)$,
it is easy to see that the sequence $\{r_{k}[u](\underline{x})\}_{k=k_{x}}^{\infty}$ is nonincreasing.
Finally, raking into account \eqref{eqq.different_level_sets}, we complete the proof.
\end{proof}

\textbf{Important notation.}
Given $\overline{u} \in \operatorname{LIP}(\operatorname{X},\operatorname{Y})$ and $x \in \operatorname{X}$ with $\overline{\operatorname{e}}_{1}[\overline{u}](x) > 0$, we let $\delta_{\overline{u}}(x)$ denote the supremum of all $\delta > 0$ such that for each $r \in (0,\delta)$ it holds
\begin{equation}
\label{eqq.delta}
\frac{1}{2}\underline{\operatorname{e}}_{1}[\overline{u}](x) \le \operatorname{e}_{1,r}[\overline{u}](x) \le 2\overline{\operatorname{e}}_{1}[\overline{u}](x), \quad \frac{1}{2}\operatorname{lip}[\overline{u}](x) \le \operatorname{var}_{r}[\overline{u}](x) \le 2\operatorname{Lip}[\overline{u}](x).
\end{equation}
Similarly, given $u \in W_{p}^{1}(\operatorname{X},\operatorname{Y})$, if $\overline{\operatorname{e}}_{1}[u](x) \in (0,+\infty)$, then $\delta_{u}(x)$ denotes the supremum of all $\delta > 0$ such that for each $r \in (0,\delta)$ it holds
\begin{equation}
\label{eqq.delta'}
\frac{1}{2}\underline{\operatorname{e}}_{1}[u](x) \le \operatorname{e}_{1,r}[u](x) \le 2\overline{\operatorname{e}}_{1}[u](x).
\end{equation}

The following observation is a keystone technical tool for Theorem \ref{Th.main_first}. We recall notation
\eqref{eqq.image_distance}, \eqref{eqq.image_distance'}, \eqref{eqq.minimal_radius}, and \eqref{eqq.Lebesgue_sets_of_differential}. 
\begin{Prop}
\label{Prop.single_Lipschitz}
Let $c_{1}:=5\operatorname{K}_{p}$, $c_{2}:=643(\operatorname{K}_{p})^{2}$. Let $u \in W^{1}_{p}(\operatorname{X},\operatorname{Y})$, $E \subset \operatorname{X}$  be such that:
\begin{itemize}
\item[\((1)\)] $E \subset \mathcal{D}_{p}[u]$ and $\mu(E) > 0$;

\item[\((2)\)] $u$ has a representative $\underline{u}$ such that $\underline{u}(\operatorname{X})$ is separable and $\underline{u}|_{E} \in \operatorname{LIP}(E,\operatorname{Y})$.
\end{itemize}
Then, there exists a set $\underline{E} \subset E$ with $\mu(E \setminus \underline{E})=0$ such that for each $x \in \underline{E}$ there exists a number $k(u,x) \in \mathbb{N}$ such that
the inequality
\begin{equation}
\label{eqq.single_Lipschitz}
\overline{\operatorname{e}}_{p}[u](x) \le c_{2}\operatorname{e}_{1,r_{k}[u](x)}[h \circ u](x)
\end{equation}
holds for any $k \geq k(u,x)$, $y \in \operatorname{Y}$, and $h \in \operatorname{LIP}_{1}(\operatorname{Y})$ satisfying the following conditions:

\begin{itemize}
\item[\(($C1$)\)] $h \equiv 2^{-k-4}$ on $B_{2^{-k-4}}(y)$ and $h \equiv 0$ on $B_{c_{1}2^{-k}}(y) \setminus B_{2^{-k-3}}(y)$;

\item[\(($C2$)\)] $x \in \operatorname{S}_{k}[u] \cap  u^{-1}(B_{2^{-k-4}}(y))$.
\end{itemize}
\end{Prop}

\begin{proof}
In view of Lemma \ref{Lm.Frechet} replacing $\underline{u}$ by $\iota \circ \underline{u}$, without loss of generality we may assume that $\operatorname{Y}$ is a complete separable subset
of $l_{\infty}$ and $\underline{y}=0$. Let $\overline{u} \in \operatorname{LIP}(\operatorname{X},l_{\infty})$
be the extension of $\underline{u}$ from $E$ to $\operatorname{X}$ given by Lemma \ref{Lm.McShane_Whitney}.
We split the rest of the proof into several steps.

\textit{Step 1.}
By Proposition \ref{Prop.different_approaches} we may assume that $u \in \operatorname{KS}^{1}_{p}(\operatorname{X},l_{\infty})$. Hence,
by Remark \ref{Rem.differentiation_integarl}, Propositions \ref{Prop.locality}, \ref{Prop.minimal_radius_tends_to_zero}, Lemma \ref{Lm.exhaustive}, Theorem \ref{Th.energy_density}, and 
Corollary \ref{Ca.Sobolev_energy_densities_comparison}
there is $\underline{E} \subset E$ such that:
\begin{itemize}

\item[\(($E1$)\)] $\mu(E \setminus \underline{E})=0$ and 
\begin{equation}
\notag
\lim\limits_{r \to +0}\frac{\mu(B_{r}(x)\setminus E)}{\mu(B_{r}(x))}=0 \quad \hbox{for all $x \in \underline{E}$};
\end{equation}

\item[\(($E2$)\)] $\underline{E} \subset \operatorname{X} \setminus N_{1}[u]$, $G_{1}[u](x) < +\infty$ for all $x \in \underline{E}$, and 
\begin{equation}
\notag
\lim\limits_{r \to +0}\frac{1}{\mu(B_{r}(x))}\int\limits_{B_{r}(x)\setminus E}G_{1}[u](y)\,d\mu(y)=0 \quad \hbox{for all $x \in \underline{E}$};
\end{equation}

\item[\(($E3$)\)] $\underline{E} \subset \bigcup_{n \in \mathbb{N}}\bigcap_{k \geq n}\operatorname{S}_{k}[u]$ and $\lim_{k \to +\infty}r_{k}[u](x) = 0$ 
for all $x \in \underline{E}$;

\item[\(($E4$)\)] for each $x \in \underline{E}$ it holds 
\begin{equation}
\notag
\varlimsup\limits_{r \to +0}|\operatorname{e}_{1,r}[u](x)-\operatorname{e}_{1,r}[\overline{u}](x)|=0;
\end{equation}

\item[\(($E5$)\)] for each $x \in \underline{E}$ it holds
\begin{equation}
\notag
\operatorname{Lip}[\overline{u}](x) \le \operatorname{K}_{p}\underline{\operatorname{e}}_{1}[\overline{u}](x), \quad   0 < \overline{\operatorname{e}}_{p}[u](x)=\overline{\operatorname{e}}_{p}[\overline{u}](x) \le
\operatorname{K}_{p}\underline{\operatorname{e}}_{1}[\overline{u}](x)=\operatorname{K}_{p}\underline{\operatorname{e}}_{1}[u](x).
\end{equation}
\end{itemize}

\textit{Step 2.} We fix $x \in \underline{E}$. To the of the proof we assume that $k \in \mathbb{N}$ and $y \in \operatorname{Y}$ are such that
\begin{equation}
\label{eqq.forgotten_6.6}
x \in \operatorname{S}_{k}[u] \cap  \underline{u}^{-1}(B_{2^{-k-4}}(y)) \cap \underline{E}.
\end{equation}
We put $r_{k}:=r_{k}[u](x)$ for brevity and consider the set
\begin{equation}
\label{eqq.forgotten_6.7}
E_{k}(x):=\Bigl\{x' \in B_{r_{k}}(x):\operatorname{d}(\underline{u}(x'),\underline{u}(x)) > \frac{1}{2}r_{k}\operatorname{e}_{1,r_{k}}[u](x)\Bigr\}.
\end{equation}
By Lemma \ref{Lm.average} we clearly have
\begin{equation}
\label{eqq.5.5}
\frac{1}{\mu(B_{r_{k}}(x))}\int\limits_{E_{k}(x)}\operatorname{d}(\underline{u}(x'),\underline{u}(x))\,d\mu(x') \geq \frac{1}{2}r_{k}\operatorname{e}_{1,r_{k}}[u](x).
\end{equation}

\textit{Step 3.} Since $x \in \underline{E}$ and $\underline{E} \subset \operatorname{X} \setminus N_{1}[u]$, by Definition \ref{Def.Hajlasz_Sobolev} we get
\begin{equation}
\label{eqq.forgotten_6.12}
\begin{split}
&\Sigma_{k,1}(x):=\frac{1}{\mu(B_{r_{k}}(x))}\int\limits_{E_{k}(x)\setminus E}\frac{\operatorname{d}(\underline{u}(x'),\underline{u}(x))}{r_{k}}\,d\mu(x')\\
&\le \frac{\mu(B_{r_{k}}(x)\setminus E)}{\mu(B_{r_{k}}(x))}G_{1}[u](x)+\frac{1}{\mu(B_{r_{k}}(x))}\int\limits_{B_{r_{k}}(x)\setminus E}G_{1}[u](y)\,d\mu(y).
\end{split}
\end{equation}
It follows from $(\operatorname{E}1)$ -- $(\operatorname{E}4)$ that letting $\Sigma_{k,2}(x):=|\operatorname{e}_{1,r_{k}}[\overline{u}](x)-\operatorname{e}_{1,r_{k}}[u](x)|$, we have
\begin{equation}
\notag
\Sigma_{k,1}(x)+\Sigma_{k,2}(x) \to +0, \quad k \to +\infty.
\end{equation}
By $(\operatorname{E}5)$, for each $\varepsilon > 0$ there is $k(u,x,\varepsilon) \in \mathbb{N}$ such that
\begin{equation}
\label{eqq.forgotten_6.16}
\Sigma_{k,1}(x)+\Sigma_{k,2}(x) \le \varepsilon\underline{\operatorname{e}}_{1}[u](x) \le \varepsilon\operatorname{e}_{1,r_{k}}[u](x) \quad \hbox{for all $k \geq k(u,x,\varepsilon)$}.
\end{equation}

\textit{Step 4.}
By $(\operatorname{E}5)$, taking into account \eqref{eqq.delta}, we get for all large enough $k \in \mathbb{N}$
\begin{equation}
\label{eqq.5.7}
\frac{\operatorname{d}(\overline{u}(x'),\overline{u}(x))}{r_{k}} \le  2\operatorname{Lip}[\overline{u}](x) \le 4\operatorname{K}_{p}\operatorname{e}_{1,r_{k}}[\overline{u}](x) \quad \hbox{for all $x' \in B_{r_{k}}(x)$}.
\end{equation}
Hence, by \eqref{eqq.forgotten_6.16} we get existence of $k_{1}=k_{1}(u,x)$ such that for each $k \geq k_{1}$
\begin{equation}
\label{eqq.forgotten_6.17}
\frac{\operatorname{d}(\overline{u}(x'),\overline{u}(x))}{r_{k}} \le 4\operatorname{K}_{p}(\operatorname{e}_{1,r_{k}}[u](x)+\Sigma_{k,2}(x)) \le 5\operatorname{K}_{p}\operatorname{e}_{1,r_{k}}[u](x).
\end{equation}

\textit{Step 5.}
The key point which follows immediately from \eqref{eqq.image_distance'}, \eqref{eqq.minimal_radius} and \eqref{eqq.forgotten_6.7} is that
\begin{equation}
\notag
\operatorname{d}(\underline{u}(x'),\underline{u}(x)) > 2^{-k-2} \quad \hbox{for all $x' \in E_{k}(x)$.}
\end{equation}
Consequently, it follows from \eqref{eqq.forgotten_6.6} and the triangle inequality that
\begin{equation}
\notag
\operatorname{d}(\underline{u}(x'),y) \geq \operatorname{d}(\underline{u}(x'),\underline{u}(x)) - \operatorname{d}(y,\underline{u}(x)) > 2^{-k-2} - 2^{-k-4} \geq 2^{-k-3} \quad \hbox{for all $x' \in E_{k}(x)$.}
\end{equation}
Finally, since $\underline{u}|_{E}=\overline{u}|_{E}$ by \eqref{eqq.forgotten_6.17} we obtain the crucial inclusion
\begin{equation}
\label{eqq.forgotten_6.15}
\underline{u}(x') \in  B_{c_{1}2^{-k}}(y) \setminus \overline{B}_{2^{-k-3}}(y) \quad \hbox{for all $x' \in E_{k}(x) \cap E$.}
\end{equation}
As a result, combining $(\operatorname{C}1)$ with \eqref{eqq.forgotten_6.15} we deduce
\begin{equation}
\label{eqq.5.6}
2^{-k-4}=h(\underline{ u}(x))=|h(\underline{u}(x'))-h(\underline{ u}(x))| \geq \frac{r_{k}}{32}\operatorname{e}_{1,r_{k}}[u](x) \quad \hbox{for all $x' \in E_{k}(x) \cap E$.}
\end{equation}

\textit{Step 6.}
Given $k \geq k_{1}$, combining  \eqref{eqq.forgotten_6.17} and \eqref{eqq.5.6} we obtain
\begin{equation}
\begin{split}
\label{eqq.6.9''}
&\frac{1}{\mu(B_{r_{k}}(x))}\int\limits_{E_{k}(x) \cap E}\frac{\operatorname{d}(\underline{u}(x),\underline{u}(x'))}{r_{k}}\,d\mu(x')\\
&\le \frac{160\operatorname{K}_{p}}{\mu(B_{r_{k}}(x))}\int\limits_{E_{k}(x) \cap E}\frac{|h(\underline{u}(x))-h(\underline{u}(x'))|}{r_{k}}\,d\mu(x') \le 
160\operatorname{K}_{p}\operatorname{e}_{1,r_{k}}[h \circ \underline{u}](x).\\
\end{split}
\end{equation}
As a result, by \eqref{eqq.5.5} and \eqref{eqq.forgotten_6.12} we get
\begin{equation}
\label{eqq.forgotten_6.18}
\frac{1}{2}\operatorname{e}_{1,r_{k}}[u](x) \le 160\operatorname{K}_{p}\operatorname{e}_{1,r_{k}}[h \circ \underline{u}](x) + \Sigma_{k,1}(x).
\end{equation}
Hence, by \eqref{eqq.forgotten_6.16} we deduce existence of $k_{2}=k_{2}(u,x) \geq k_{1}$ such that for each $k \geq k_{2}$
\begin{equation}
\label{eqq.forgotten_6.19}
\operatorname{e}_{1,r_{k}}[u](x) \le 321\operatorname{K}_{p}\operatorname{e}_{1,r_{k}}[h \circ \underline{u}](x).
\end{equation}

\textit{Step 7.}
If $k \in \mathbb{N}$ is so large that $r_{k}=r_{k}[u](x) < \min\{\delta_{\overline{u}}(x),\delta_{u}(x)\}$ and $k \geq k_{2}$, then by $(\operatorname{E}3)$, $(\operatorname{E}5)$, \eqref{eqq.delta'}, and \eqref{eqq.forgotten_6.19} we obtain
\begin{equation}
\begin{split}
\label{eqq.6.11''}
&\overline{\operatorname{e}}_{p}[u](x) \le 2\operatorname{K}_{p}\operatorname{e}_{1,r_{k}}[\overline{u}](x) \le 643(\operatorname{K}_{p})^{2}\operatorname{e}_{1,r_{k}}[h \circ \underline{u}](x).
\end{split}
\end{equation}
The proof is complete.
\end{proof}

\section{Proof of the first main result}

Now we are ready to prove the first main result of this paper.

\textit{proof of Theorem \ref{Th.main_first}.} Taking into account Proposition \ref{Prop.Poincare_forgotten} and Remark \ref{Rem.Sobolev_forgotten}
without loss of generality we may assume that the metric space $(\operatorname{X},\rho)$ is geodesic at the scale $1$.
To the end of the proof we fix $u \in W_{p}^{1}(\operatorname{X},\operatorname{Y})$, $k \in \mathbb{N}$ and $r \in (0,1)$.
We also fix a Borel representative $\underline{u}$ of $u$ such that $\underline{u}(\operatorname{X})$ is separable.
Furthermore, taking into account Lemma \ref{Lm.Frechet} we may assume that  $\underline{u}(\operatorname{X})$ is a subset of $l_{\infty}$.
The rest of the proof will be split into several steps.

\textit{Step 1.} Using Proposition \ref{Prop.Lusin-Lipschitz}, we fix an arbitrary set $V \subset \mathcal{D}_{p}[u]$ such that $\mu(V) \in (0,+\infty)$, $\underline{u}|_{V}$ is Lipschitz and
\begin{equation}
\notag
\int\limits_{\operatorname{X}}(|Du|_{p}(x))^{p}\,d\mu(x) < 2\int\limits_{V}(|Du|_{p}(x))^{p}\,d\mu(x).
\end{equation}
By Lemma \ref{Lm.McShane_Whitney} there is $\overline{u} \in \operatorname{LIP}(\operatorname{X},l_{\infty})$ such that $\overline{u}|_{V}=\underline{u}|_{V}$.

\textit{Step 2.} For each $l \in \mathbb{N}$, we put (here $\widetilde{\lambda}$ is the same as in Proposition \ref{Prop.maximal_function})
\begin{equation}
\notag
E_{l}[u]:= \operatorname{S}_{l}[u] \cap \Bigl\{x \in V:r_{l}[u](x) < \frac{r}{\widetilde{\lambda}}\Bigr\},
\end{equation}
where the parameter $\lambda$ is the same as in \eqref{eqq.Poincare}.

Since $u \in W_{p}^{1}(\operatorname{X},\operatorname{Y})$, by Lemma \ref{Lm.auxiliary_norm} we have
\begin{equation}
\notag
\mu(\mathcal{D}_{p}[u] \setminus \{x \in \operatorname{X}:\overline{\operatorname{e}}_{p}[u](x) > 0\})=0.
\end{equation}
Hence, by Lemma \ref{Lm.exhaustive}, Proposition \ref{Prop.minimal_radius_tends_to_zero} and Egorov's theorem (we recall that $\mu(V) < +\infty$) we take and fix $\underline{l} \in \mathbb{N}$
so large that
\begin{equation}
\notag
\int\limits_{\operatorname{X}}(|Du|_{p}(x))^{p}\,d\mu(x) < 2\int\limits_{E}(|Du|_{p}(x))^{p}\,d\mu(x),
\end{equation}
where, we set 
$$
E:=E_{\underline{l}}[u] \cap \mathcal{D}_{p}[u].
$$ 
Now let $\underline{E} \subset E$, $x \in \underline{E}$, and $k(u,x)$  have the same meaning as in Proposition \ref{Prop.single_Lipschitz}.
Clearly, there exists so large $N_{u} \in \mathbb{N}$ that
\begin{equation}
\notag
\int\limits_{\operatorname{X}}(|Du|_{p}(x))^{p}\,d\mu(x) < 2\int\limits_{U}(|Du|_{p}(x))^{p}\,d\mu(x),
\end{equation}
where $U:=\{x \in \underline{E}:k(u,x) \le N_{u}\}$.

\textit{Step 3.} We apply Proposition \ref{Prop.Reduction_to_Vitali} with $f=(|Du|_{p})^{p}$ and
\begin{equation}
\notag
\overline{f}(x):=\varlimsup\limits_{r \to +0}\fint\limits_{B_{r}(x)}f(y)\,d\mu(y).
\end{equation}
We obtain that the measure
$\mathfrak{m}:=\overline{u}_{\sharp}((|Du|_{p})^{p}\chi_{U}\mu)$ is uniformly weakly asymptotically doubling on the space $\operatorname{Z}:=\overline{u}(\underline{U},\operatorname{d})$ (the set $\underline{U} \subset U$
has the same meaning as in \eqref{eqq.5.4''}). Taking into account
\eqref{eqq.asymptotically_doubling} with $\tau = 180c_{1}$ (where $c_{1}$ is the same as in Proposition \ref{Prop.single_Lipschitz}), we get
existence of a constant $C > 0$ such that,
for each $z \in \operatorname{Z}$ there exists a sequence of balls $\{B_{s}(z)\}_{s=1}^{\infty}$ in $\operatorname{Y}$ (initially we get those balls in $\operatorname{Z}$ but we then 
consider them as balls in $\operatorname{Y}$) satisfying
\begin{equation}
\label{eqq.7.2''}
\lim\limits_{s \to \infty}r(B_{s}(z))=0 \quad \hbox{and} \quad \sup\limits_{z \in \operatorname{Z}}\sup\limits_{s \in \mathbb{N}}\frac{\mathfrak{m}(\tau B_{s}(z))}{\mathfrak{m}(B_{s}(z))} \le C.
\end{equation}
Hence, for each $z \in \operatorname{Z}$ we can fix a ball $\underline{B}(z)$ in $\operatorname{Y}$ such that the following properties hold:
\begin{itemize}
\item[\(($B1$)\)] $r(\underline{B}(z)) = 2^{-k(z)}$ for some $k(z) \in \mathbb{N}$;

\item[\(($B2$)\)] $k(z) \geq \max\{\underline{l}+5,N_{u}+4\}$;

\item[\(($B3$)\)] $r(\underline{B}(z)) \in [r(B_{s}(z)),2r(B_{s}(z)))$ for some $s \in \mathbb{N}$.
\end{itemize}
It follows immediately from \eqref{eqq.7.2''} and $(\operatorname{B3})$ that
\begin{equation}
\label{eqq.7.3''}
\sup\limits_{z \in \operatorname{Z}}\frac{\mathfrak{m}(\frac{\tau}{2}\underline{B}(z))}{\mathfrak{m}(\underline{B}(z))} \le C.
\end{equation}

\textit{Step 4.} We put $\mathcal{B}:=\{\frac{\tau}{10}\underline{B}(z):z \in \operatorname{Z}\}$. Applying the $5B$-Vitaly covering lemma to the family $\mathcal{B}$, we get an at most countable family of balls $\{\underline{B}_{m}\}_{m \in \mathcal{I}}$ with $\mathcal{I} \subset \mathbb{N}$ such that the family $\{\frac{\tau}{10}\underline{B}_{m}\}_{m \in \mathcal{I}}$ is disjoint and $\bigcup_{m \in \mathcal{I}}\frac{\tau}{2}\underline{B}_{m} \supset \operatorname{Z}$. Hence, letting
\begin{equation}
\notag
\operatorname{C}_{1}:=2C \quad \hbox{and} \quad
\mathcal{S}_{u,r}:=\overline{u}^{-1}\Bigl(\bigcup_{m \in \mathcal{I}}\underline{B}_{m}\Bigr) \cap U,
\end{equation}
we get the required estimate
\begin{equation}
\label{eqq.good_inequality}
\int\limits_{\operatorname{X}}(|Du|_{p}(x))^{p}\,d\mu(x) \le \operatorname{C}_{1}\int\limits_{\mathcal{S}_{u,r}}(|Du|_{p}(x))^{p}\,d\mu(x).
\end{equation}

\textit{Step 5.}
 Now we construct our special test function $\underline{h}_{u,r}$ as follows.
For each $m \in \mathcal{I}$, we set
\begin{equation}
\widetilde{\underline{h}}_{u,r,m}(y):=
\begin{cases}
r(\underline{B}_{m}), \quad \hbox{if $y \in \underline{B}_{m}$};\\
0, \quad \hbox{if $y \notin 2\underline{B}_{m}$}.
\end{cases}
\end{equation}
Then we put
\begin{equation}
\notag
\widetilde{\underline{h}}_{u,r}(y):=\sum\limits_{m \in \mathcal{I}}\widetilde{\underline{h}}_{u,r,m}(y), \quad y \in \operatorname{Y} \setminus \bigcup\limits_{m \in \mathcal{I}}(2\underline{B}_{m}\setminus \underline{B}_{m}).
\end{equation}
By the construction, given $m,m' \in \mathcal{I}$ with $m \neq m'$, we have
$$
\operatorname{dist}(\underline{B}_{m},\underline{B}_{m'}) > \max\{r(\underline{B}_{m}),r(\underline{B}_{m'})\}.
$$
This clearly gives
\begin{equation}
\notag
\widetilde{\underline{h}}_{u,r} \in \operatorname{LIP}_{1}\Bigl(\operatorname{Y} \setminus \bigcup\limits_{m}(2\underline{B}_{m}\setminus \underline{B}_{m})\Bigr).
\end{equation}
Using Lemma \ref{Lm.McShane_Whitney}
we get the extension $\underline{h}_{u,r} \in \operatorname{LIP}_{1}(\operatorname{Y})$ of $\widetilde{\underline{h}}_{u,r}$ to the whole space $\operatorname{Y}$.

\textit{Step 6.} By the construction, $r(\underline{B}_{m})=2^{-k_{m}}$ for some $k_{m} \in \mathbb{N}$ with $k_{m} \geq \underline{l}+5$.
We put $r_{x,m}:=r_{k_{m}-4}[u](x)$ for each $x \in \mathcal{S}_{u,r}$ and note that $\widetilde{\lambda} r_{x,m} < r$.
Let $q=q(1) \in (1,p)$ be the same as in Proposition \ref{Prop.Sobolev_imb}.
Take into account that, given $m$, we have
\begin{equation}
\notag
\hbox{$\underline{h}_{u,r} \equiv 1$ on $\underline{B}_{m}$} \quad \hbox{and} \quad \hbox{$\underline{h}_{u,r} \equiv 0$ on $16c_{1}\underline{B}_{m} \setminus 2\underline{B}_{m}$}.
\end{equation}
As a result, applying
Proposition \ref{Prop.single_Lipschitz} for each ball $\underline{B}_{m}$, $m \in \mathbb{N}$ with $h$ replaced by $\underline{h}_{u,r}$ and $B_{2^{-k-4}}(y)$ replaced by $\underline{B}_{m}$, we derive the required estimate
\begin{equation}
\label{eqq.crucial_inequality}
\begin{split}
&\overline{\operatorname{e}}_{p}[u](x) \le C\operatorname{e}_{1,r_{x,m}}[\underline{h}_{u,r} \circ u](x)\\
&\le C\operatorname{M}^{\le \widetilde{\lambda} r_{x,m}}_{q}[|D(\underline{h}_{u,r}\circ u)|_{p}](x) \le C \operatorname{M}^{\le r}_{q}[|D(\underline{h}_{u,r}\circ u)|_{p}](x) \quad \hbox{for $\mu$-a.e. $x \in \mathcal{S}_{u,r}$}.
\end{split}
\end{equation}
If $\underline{h}_{u,r}(\underline{y})\neq 0$, then one should replace $\underline{h}_{u,r}$ by $\underline{h}_{u,r}-\underline{h}_{u,r}(\underline{y})$ retaining the same notation and
note this does not affect to the validity of \eqref{eqq.crucial_inequality}. The proof is complete.
\hfill$\Box$

Combining Theorem \ref{Th.main_first} with Proposition \ref{Prop.maximal_function} we immediately get the following important observation.
\begin{Ca}
\label{Ca.main}
For each $p \in (1,\infty)$ and each $\operatorname{X}=(\operatorname{X},\rho,\mu) \in \mathcal{A}_{p}$, there exists a constant $\operatorname{C}_{3} > 0$ such that the following holds.
For each complete pointed metric space $\operatorname{Y}=(\operatorname{Y},\operatorname{d},\underline{y})$ and each $u \in W_{p}^{1}(\operatorname{X},\operatorname{Y})$,
there is a function $\underline{h} \in \operatorname{LIP}_{1}(\operatorname{Y},\underline{y})$ such that
\begin{equation}
\label{eqq.ca_main}
\||Du|_{p}\|_{L_{p}(\operatorname{X})} \le \operatorname{C}_{3} \||D(\underline{h} \circ u)|_{p}\|_{L_{p}(\operatorname{X})}.
\end{equation}
\end{Ca}

\section{A weak form of the second main result}

Throughout the section the following data will be fixed:
\begin{itemize}
\item[\(($D.7.1$)\)] a parameter $p \in (1,\infty)$;

\item[\(($D.7.2$)\)] a metric measure space $\operatorname{X}=(\operatorname{X},\operatorname{d},\mu) \in \mathcal{A}_{p}$;

\item[\(($D.7.3$)\)] a pointed complete metric space $\operatorname{Y}=(\operatorname{Y},\operatorname{d},\underline{y})$.
\end{itemize}

Given a nonempty set $\operatorname{Z} \subset \operatorname{Y}$, we say that $\mathcal{N} \subset \operatorname{Z}$ is an \textit{$\epsilon$-separated subset} of $\operatorname{Z}$ if $\operatorname{d}(z,z') \geq \epsilon$ for all $z,z' \in \mathcal{N}$ with $z \neq z'$.
An $\epsilon$-separated subset $\mathcal{N}$ of $\operatorname{Z}$ is said to be \textit{maximal} if
\begin{equation}
\label{eqq.epsilon_separated}
\operatorname{dist}(z,\mathcal{N}):=\inf\{\operatorname{d}(z,z'):z' \in \mathcal{N}\} < \epsilon \quad \hbox{for every $z \in \operatorname{Z}$.}
\end{equation}

\textbf{Important notation.} 
Given a compact set $\operatorname{Z} \subset \operatorname{Y}$ and $k \in \mathbb{Z}$ with $2^{-k} \le 2\operatorname{diam}\operatorname{Z}$, we fix a maximal $2^{-k}$-separated subset $\{z_{k,m}\}_{m=1}^{N_{k}}$, $N_{k} \in \mathbb{N}$ of $\operatorname{Z}$
and put
\begin{equation}
\label{eqq.resol_unity_2}
\psi_{k,m}(\cdot):=\max\Bigl\{0,2^{-k}-\operatorname{dist}(\cdot,\overline{B}_{2^{-k}}(z_{k,m}))\Bigr\}, \quad m \in \{1,...,N_{k}\}.
\end{equation}
We define the map $\mathfrak{I}_{k}[\operatorname{Z}]:\operatorname{Y} \to \mathbb{R}^{N_{k}}_{\infty}$ by the equality
\begin{equation}
\label{eqq.imbed_1}
\mathfrak{I}_{k}[\operatorname{Z}](y):=(\psi_{k,1}(y)-\psi_{k,1}(\underline{y}),...,\psi_{k,N_{k}}(y)-\psi_{k,N_{k}}(\underline{y})), \quad y \in \operatorname{Y}.
\end{equation}
Now we introduce the crucial concept. Ideologically it is close to the famous Fr\'echet imbedding (compare with Lemma \ref{Lm.Frechet}). On the other hand, it has nice
local behaviour. Given a compact set $\operatorname{Z} \subset \operatorname{Z}$ we let $\underline{k}$ denote the minimal among all $k \in \mathbb{Z}$ such that $2^{-k} \le 2\operatorname{diam}\operatorname{Z}$.
\begin{Def}
\label{Def.smooth_Kuratowski}
Given a compact set $\operatorname{Z} \subset \operatorname{Y}$,
by a smooth localized Fr\'echet--Kuratowski-type imbedding we mean a map $\mathfrak{I}[\operatorname{Z}]:\operatorname{Y} \to l_{\infty}$ defined by the equality
\begin{equation}
\label{eqq.imbed_2}
\mathfrak{I}[\operatorname{Z}](y):=(\mathfrak{I}_{\underline{k}}[\operatorname{Z}](y),....,\mathfrak{I}_{l}[\operatorname{Z}](y),...), \quad y \in \operatorname{Y}.
\end{equation}
Given $n \in \mathbb{N}$, we define the $n$-th offcut of $\mathfrak{I}[\operatorname{Z}]$ by letting
\begin{equation}
\label{eqq.imbed_3}
\mathfrak{I}^{n}[\operatorname{Z}](y):=(\mathfrak{I}_{\underline{k}}[\operatorname{Z}](y),...,\mathfrak{I}_{n}[\operatorname{Z}](y),0,...,0,...), \quad z \in \operatorname{Z}.
\end{equation}
\end{Def}

\begin{Remark}
\label{Rem.Z_dependence}
In fact the map $\mathfrak{I}[\operatorname{Z}]$ depends both on $\operatorname{Z}$ and the choice of family $\{\psi_{k,m}\}$.
We use the heavy notation and retain the symbol $\operatorname{Z}$ since it will be convenient for the next subsection, where
we will work with not a single compact metric space $\operatorname{Z}$ but rather with the family of such spaces.
The possible dependence on the choice of $\{\psi_{k,m}\}$ will not play any essential role.
\end{Remark}

The following assertion justifies the introducing of $\mathfrak{I}[Z]$ in a sense.
\begin{Prop}
\label{Prop.almost_isometry}
Given a compact set $\operatorname{Z} \subset \operatorname{Y}$, the map $\mathfrak{I}[Z]$ has the following properties:

\begin{itemize}
\item[\((1)\)] $\mathfrak{I}[Z] \in \operatorname{LIP}_{1}(\operatorname{Y},l_{\infty})$ and $\mathfrak{I}[Z](\underline{y})=0$;

\item[\((2)\)] for each pair $(z_{1},z_{2}) \in \operatorname{Z} \times \operatorname{Z}$  it holds
\begin{equation}
\label{eqq.almost_isometry}
\frac{1}{8}\operatorname{d}(z_{1},z_{2}) \le \|\mathfrak{I}[\operatorname{Z}](z_{1})-\mathfrak{I}[\operatorname{Z}](z_{2})\|_{\infty} \le \operatorname{d}(z_{1},z_{2}).
\end{equation}
\end{itemize}
\end{Prop}

\begin{proof}
Since $\{\psi_{k,m}\} \subset \operatorname{LIP}_{1}(\operatorname{Z})$ for each $k \in \mathbb{Z}$, the right inequality in \eqref{eqq.almost_isometry} follows easily.
To prove the left inequality, given $z_{1},z_{2}$ we take $k \in \mathbb{Z}$ to be the minimal number for which $2^{-k} \le \operatorname{d}_{\operatorname{Z}}(z_{1},z_{2})$.
By the maximality of the net $\{z_{k,m}\}$ there is $m \in \{1,...,N_{k+2}\}$ such that  $\psi_{k,m}(z_{1})=2^{-k-2}$ and $\psi_{k,m}(z_{2})=0$.
\end{proof}

The following assertion is an immediate consequence of \eqref{eqq.resol_unity_2} -- \eqref{eqq.imbed_3}.
\begin{Lm}
\label{Lm.finite_dimensional}
For each  compact set $\operatorname{Z} \subset \operatorname{Y}$ and each $n \in \mathbb{N}$, it holds
\begin{equation}
\label{eqq.finite_dimensional}
\sup\limits_{y \in \operatorname{Y}}\|\mathfrak{I}^{n}[\operatorname{Z}](y)-\mathfrak{I}[\operatorname{Z}](y)\| < \frac{1}{2^{n}}.
\end{equation}
\end{Lm}

Now we are ready to prove the central result of this section.

\begin{Th}
\label{Th.weak_main}
Let $u \in L_{p}(\operatorname{X},\operatorname{Y})$ be such that $h\circ u \in W_{p}^{1}(\operatorname{X})$ for every $h \in \operatorname{LIP}_{1}(\operatorname{Y},\underline{y})$.
Then, $\mathfrak{I}[\operatorname{Z}] \circ u \in W_{p}^{1}(\operatorname{X},l_{\infty})$ for each compact set $\operatorname{Z} \subset \operatorname{Y}$.
\end{Th}

\begin{proof}
We fix an arbitrary compact $\operatorname{Z} \subset \operatorname{Y}$. We put $\overline{u}:=\mathfrak{I}[\operatorname{Z}] \circ u$ and $\overline{u}^{n}:=\mathfrak{I}^{n}[\operatorname{Z}] \circ u$, $n \in \mathbb{N}$.
By (1) of Proposition \ref{Prop.almost_isometry} we clearly have $h \circ \overline{u} \in W_{p}^{1}(\operatorname{X})$ for each $h \in \operatorname{LIP}_{1}(l_{\infty})$
with $h(0)=0$. It follows from Lemma \ref{Lm.finite_dimensional} that $\overline{u}^{n}(x) \to \overline{u}(x)$ for $\mu$-a.e. $x \in \operatorname{X}$.
Given $n \in \mathbb{N}$, we have $\mathfrak{I}^{n}[\operatorname{Z}] \in \operatorname{LIP}_{1}(\operatorname{Y},\mathbb{R}^{N_{n}})$ and consequently $\overline{u}^{n} \in W_{p}^{1}(\operatorname{X},\mathbb{R}^{N_{n}})$. Using Corollary \ref{Ca.main},  we find a function $\underline{h}_{n} \in \operatorname{LIP}_{1}(\mathbb{R}^{N_{n}},0)$ such that
\begin{equation}
\notag
\||D\overline{u}_{n}|_{p}\| \le \operatorname{C}_{3}\||D(\underline{h}_{n} \circ \overline{u})|_{p}\|_{L_{p}(\operatorname{X})}.
\end{equation}
Given $n \in \mathbb{N}$, we let $\pi_{n}$ denote the natural projection of $l_{\infty}$ onto $\mathbb{R}^{N_{n}}$. Then we have $h_{n}:=\underline{h}_{n}\circ \pi_{n} \in \operatorname{LIP}_{1}(l_{\infty},0)$.
Now application of Lemma \ref{Lm.Oleinik_Pavlov} gives
\begin{equation}
\notag
\sup\limits_{n \in \mathbb{N}}\||D\overline{u}_{n}|_{p}\| \le \operatorname{C}_{3}\sup\limits_{n \in \mathbb{N}}\||D(h_{n} \circ \overline{u})|_{p}\|_{L_{p}(\operatorname{X})} < +\infty.
\end{equation}
Hence, by Proposition \ref{Prop.weak_convergence} the claim follows easily.
\end{proof}

\section{Proof of the second main result}

Now we are ready to prove the \textit{second main} result of this paper.

\textit{proof of Theorem \ref{Th.main_second}.} It will be convenient to split the proof into several simple steps.

\textit{Step 1.} First of all we note that $u \in L_{p}(\operatorname{X},\operatorname{Y})$. Indeed, it is sufficient
to compose $\underline{u}$ with $h:=\operatorname{d}(\cdot,\underline{y})$.
Now we put $f:=(\operatorname{d}(\cdot,\underline{y}))^{p}$ and consider the measure $\mathfrak{m}:=\underline{u}_{\sharp}(f\mu)$. Since $\mathfrak{m}$ is a finite measure on $\operatorname{Y}$, using the Ulam theorem
we find a sequence $\{\operatorname{Z}_{n}\} \subset \operatorname{Y}$ of compact sets such that $\operatorname{Z}_{n} \subset \operatorname{Z}_{n+1}$ for all $n \in \mathbb{N}$ and $\mathfrak{m}(\operatorname{Y}\setminus \operatorname{Z}_{n}) \to 0$, $n \to \infty$. Without loss of generality we will assume that $\underline{y} \in \operatorname{Z}_{n}$ for all $n \in \mathbb{N}$. Hence, 
\begin{equation}
\mu\Bigl((\operatorname{X}\setminus \bigcup\limits_{n \in \mathbb{N}}\underline{u}^{-1}(\operatorname{Z}_{n})\Bigr) = 0.
\end{equation}

\textit{Step 2.}
By Theorem \ref{Th.weak_main} we get a sequence $\{v_{n}\} \subset W_{p}^{1}(\operatorname{X},l_{\infty})$, where
$v_{n}$ is an equivalence class of $\underline{v}_{n}:=\mathfrak{I}[\operatorname{Z}_{n}]\circ \underline{u}$, $n \in \mathbb{N}$. Using Corollary \ref{Ca.main}, we find functions $\underline{h}_{n} \in \operatorname{LIP}_{1}(l_{\infty},0)$, $n \in \mathbb{N}$, such that
\begin{equation}
\notag
\||D\underline{v}_{n}|_{p}\| \le \operatorname{C}_{3}\||D(\underline{h}_{n} \circ \underline{v}_{n})|_{p}\|_{L_{p}(\operatorname{X})}.
\end{equation}
As a result, since $h \circ \mathfrak{I}[\operatorname{Z}_{n}] \in \operatorname{LIP}_{1}(\operatorname{Y},\underline{y})$ for each $n \in \mathbb{N}$ and each $h \in \operatorname{LIP}_{1}(l_{\infty},0)$, by
Lemma \ref{Lm.Oleinik_Pavlov} we deduce
\begin{equation}
\label{eqq.9.1}
\begin{split}
&\sup\limits_{n \in \mathbb{N}}\||D\underline{v}_{n}|_{p}\|_{L_{p}(\operatorname{X})} \le \operatorname{C}_{3}\sup\limits_{n \in \mathbb{N}}\sup\limits_{h \in \operatorname{LIP}_{1}(l_{\infty},0)}\||D(h \circ \underline{v}_{n})|_{p}\|_{L_{p}(\operatorname{X})}\\
&\le \operatorname{C}_{3}\sup\limits_{h \in \operatorname{LIP}_{1}(\operatorname{Y},\underline{y})}\||D(h \circ \underline{u})|_{p}\|_{L_{p}(\operatorname{X})} < +\infty.
\end{split}
\end{equation}

\textit{Step 3.}
Taking into account Proposition \ref{Prop.different_approaches} we have
$\{v_{n}\} \subset \operatorname{HS}^{1}_{p}(\operatorname{X},l_{\infty})$. Furthermore, by Proposition \ref{Prop.Poincare_self_improvement}
there are $\underline{q} \in (1,p)$, $\underline{\lambda} \geq 1$ and $\underline{c} > 0$ such that \eqref{eqq.Poincare} holds true with $R = 1$
and $q$, $\lambda$, $\operatorname{c}$ replaced by $\underline{q}$, $\underline{\lambda}$, $\underline{\operatorname{c}}$, respectively. Hence, repeating
almost verbatim arguments from the proof of Proposition 2.17 in \cite{GT} we derive existence of a constant $\underline{C} > 0$ depending
on $\operatorname{D}_{\mu}(\underline{\lambda})$ and $\underline{\lambda}$, $\underline{\operatorname{c}}$ such that, for each $n \in \mathbb{N}$, $G_{n}:=G[v_{n}]:=\underline{C}\operatorname{M}^{\le 2\underline{\lambda}}_{\underline{q}}[|Dv_{n}|_{p}]$ is a Hajlasz gradient of $v_{n}$ at the scale $1$. Hence, by Proposition \ref{Prop.maximal_function} and \eqref{eqq.9.1} we have
\begin{equation}
\label{eqq.9.2}
\sup\limits_{n \in \mathbb{N}}\|G_{n}\|_{L_{p}(\operatorname{X})} < +\infty.
\end{equation}

\textit{Step 4.}
Given $n \in \mathbb{N}$, we put $U_{n}:=\underline{u}^{-1}(\operatorname{Z}_{n})$ and let $u_{n}$ denote the equivalence class of
$\underline{u}_{n}:=\underline{u}\chi_{U_{n}}$. As a result, taking into account Proposition \ref{Prop.almost_isometry}
we get, for each $n \in \mathbb{N}$ the existence of $E_{n} \subset U_{n}$ with $\mu(E_{n})=0$, such that
\begin{equation}
\begin{split}
\notag
&\operatorname{d}(\underline{u}_{n}(x_{1}),\underline{u}_{n}(x_{2})) \le 8\|v_{n}(x_{1})-v_{n}(x_{2})\|\\
&\le \rho(x_{1},x_{2})(G_{n}(x_{1})+G_{n}(x_{2})) \quad \hbox{for all $x_{1},x_{2} \in U_{n} \setminus E_{n}$ with $\rho(x_{1},x_{2}) \le 1$}.
\end{split}
\end{equation}

\textit{Step 5.} It follows from Lemma \ref{Lm.finite_dimensional} that $\underline{u}^{n}(x) \to \underline{u}(x)$ for $\mu$-a.e. $x \in \operatorname{X}$. 
Using \eqref{eqq.9.2} in combination with Proposition \ref{Prop.weak_convergence'} we deduce that $u \in \operatorname{HS}^{1}_{p}(\operatorname{X},\operatorname{Y})$.
Finally, taking into account Proposition \ref{Prop.different_approaches} we conclude.
\hfill$\Box$

\textbf{Acknowledgments} I am grateful to my student R.~D.~Oleinik who firstly posed Problem \ref{Problem.Oleinik} and attracted my attention to the subject.
Initially, I was a little bit sceptic about that. However, during the works \cite{OT,ST} I realised the depth and the beauty of the problem. Furthermore,
several mathematicians mentioned in the introduction appeared to be drawn into the orbit of Problem \ref{Problem.Oleinik}.

I am also grateful to my students A.~Y.~Chikalov, R.~D.~Oleinik, and to my colleague I.~M.~Vasilyev for careful reading of the preliminary version of this manuscript. Their valuable comments
helped me to improve the exposition.

\textbf{Data availability} Data sharing not applicable to this article as no datasets were generated or analysed during
the current study.

\textbf{Declarations}

\textbf{Conflict of interest} The author has no conflict of interest to declare that are relevant to the content of this
article.

\textbf{AI declaration} I confirm that AI has not been used at any stage of writing of this paper. All ideas, proofs
and verifications of the results have been carried out by \textit{myself only.}





%
%
\end{document}